\documentclass[a4paper,10pt]{amsart}
\usepackage[utf8]{inputenc}

\usepackage{graphicx,amssymb,amsmath,amsfonts,amsthm,amscd,hyperref,textcomp}
\usepackage{xcolor} 
\usepackage[style=alphabetic]{biblatex} 
\usepackage{fancyvrb}

\newtheorem{lemma}{Lemma}
\newtheorem{cor}{Corollary}
\newtheorem{thm}{Theorem}
\newtheorem{prop}{Proposition}

\newtheorem{rmk}{Remark}\theoremstyle{remark}

\newcommand{\CC}{\mathbb C}

\newcommand{\QQ}{\mathbb Q}
\newcommand{\Fp}{{\mathbb F}_p}

\newcommand{\Ql}{\bar{\mathbb Q}_\ell}

\newcommand{\FF}{\mathcal F}
\newcommand{\GGG}{\mathcal G}

\newcommand{\LL}{\mathcal L}
\newcommand{\R}{\mathrm R}

\newcommand{\AAA}{\mathbb A}

\newcommand{\Gm}{{\mathbb G}_m}
\newcommand{\Gmk}{\mathbb G_{m,k}}
\newcommand{\ZZ}{\mathbb Z}

\newcommand{\triv}{\mathbf 1}

\title{Finite monodromy of some two-parameter families of exponential sums}
\author{Francisco García-Cortés}
\address{Dpto. de \'Algebra, Fac. de Matem\'aticas \\
  Universidad de Sevilla \\
  c/Tarfia, s/n \\
  41012 Sevilla, Spain \\
  {\tt fragarcor3@alum.us.es}}
\author{Antonio Rojas-Le\'on}
\address{Dpto. de \'Algebra, Fac. de Matem\'aticas \\
  Universidad de Sevilla \\
  c/Tarfia, s/n \\
  41012 Sevilla, Spain \\
  {\tt arojas@us.es}}

\begin{document}

\begin{abstract}
We determine the set of polynomials $f(x)\in k[x]$, where $k$ is a finite field, such that the local system on $\mathbb G_m^2$ which parametrizes the family of exponential sums $(s,t)\mapsto\sum_{x\in k}\psi(sf(x)+tx)$ has finite monodromy, in two cases: when $f(x)=x^d+\lambda x^e$ is a binomial and when $f(x)=(x-\alpha)^d(x-\beta)^e$ is of Belyi type.
\end{abstract}

\maketitle

\renewcommand{\thefootnote}{}
\footnote{The first author is partially supported by grants PID2020-117843GB-I00 and PID2020-114613GB-I00 funded by MICIU/AEI/10.13039/501100011033}
\footnote{The second author is partially supported by grant PID2020-114613GB-I00 funded by MICIU/AEI/10.13039/501100011033}


\section{Introduction}

Let $k$ be a finite field of characteristic $p>0$ and cardinality $q=p^a$. We fix a prime $\ell\neq p$ and, given an algebraic variety $X$ defined over $k$, we will work in the category $\mathcal{S}h(X,\Ql)$ of constructible $\ell$-adic sheaves on $X$. Fix an algebraic closure $\bar k$ of $k$ and denote by $k_r$ de degree $r$ extension of $k$ in $\bar k$. Let $\psi:k\to\CC$ be the natural additive character $\psi(t)=\exp\left(\frac{2\pi i\mathrm{Tr}_{k/\Fp}t}{p}\right)$ and, for every $r\geq 1$, denote by $\psi_r$ its pull-back to $k_r$ obtained by composing $\psi$ with the trace map $k_r\to k$.

Let $f(x)=\sum_{i=0}^d a_ix^i \in k[x]$ be a polynomial such that there is some non-zero $i$ which is not a power of $p$ with $a_i\neq 0$. Then there exists a lisse sheaf $\FF_f$ on $U_f\times\AAA^1_k$, pure of weight 0, (where $U_f$ is a certain dense open subset of $\Gmk$) whose Frobenius trace function at a rational point $(s,t)\in U_f(k_r)\times k_r$ is given by
$$
F(k_r;s,t)=-\frac{1}{q^{r/2}}\sum_{x\in k_r}\psi_r(sf(x)+tx).
$$

Studying the monodromy group of these sheaves is interesting for its applications to coding theory and also as they give interesting examples of lisse sheaves whose monodromy group is a sporadic simple finite group \cite{KRL19}. In \cite{RL19} it was shown that, if $f(x)$ is a monomial, then the monodromy is either finite, the full symplectic group $\mathrm{Sp}_{d-1}$ (for $d$ odd) or the special linear group $\mathrm{SL}_{d-1}$ (for $d$ even) (more generally, see \cite[Theorem 10.2.4]{KT24}). In characteristic 2, the latter implies that the function $f(x)$ is not APN-exceptional \cite{AMR10}. Determining the non-finiteness of the monodromy for more general $f(x)$ can then provide additional results in this direction.

The question of which monomials $f(x)=x^d$ produce sheaves with finite monodromy has been recently fully solved by N. Katz and P.H. Tiep in \cite[Theorem 10.2.6]{KT24}:

\begin{thm}[Katz, Tiep]\label{KT1}
 Let $f(x)=x^d$, and $d_{\hat p}$ the prime-to-$p$ part of $d$. Then the sheaf $\FF_{x^d}$ has finite (geometric or arithmetic) monodromy group on $\Gmk^2$ if and only if
 \begin{enumerate}
  \item $d_{\hat p}=p^a+1$ for some $a\geq 0$ ($a>0$ if $p=2$).
  \item $p>2$ and $d_{\hat p}=\frac{p^a+1}{2}$ for some $a>0$.
  \item $d_{\hat p}=\frac{p^{ab}+1}{p^a+1}$ for some $a,b>0$ with $b$ odd (including $d_{\hat p}=1$).
  \item $p=5$ and $d_{\hat p}=7$. \end{enumerate}
\end{thm}

Katz and Tiep's theorem deals with the $d$ prime to $p$ case, from which one can easily deduce the general case using the fact that for any $f:X\to\AAA^1$ the sheaf $\LL_{\psi(f)}$, pull-back by $f$ of the Artin-Schreier sheaf $\LL_\psi$ on $\AAA^1$, is canonically isomorphic to $\LL_{\psi(f^p)}$.

We will say that a positive integer $d$ is an \emph{FM-exponent} for the prime $p$ if the local system $\FF_{x^d}$ has finite monodromy, that is, if it falls in one of the cases of the previous theorem. Note that, by by \cite[Lemma 2.5,2.6]{KRLT20}, this is equivalent to the one-parameter family with trace function
$$
(k_r;t)\mapsto -\frac{1}{q^{r/2}}\sum_{x\in k_r}\psi_r(x^d+tx).
$$
having finite monodromy. The equivalence of the arithmetic and geometric monodromy groups being finite follows from \cite[Proposition 2.1]{KRLT20}.

In this article we study two new families of polynomials for which we can find numerical criterions for finiteness of the monodromy, based on Kubert's $V$ function, and determine for which of such $f$ the local system $\FF_f$ has finite monodromy. Recall that $V$ (which we will denote by $V_p$ to make it clear that it depends on $p$) is a function with rational values in $[0,1)$ defined in $(\QQ/\ZZ)_{\text{not }p}:=\{\frac{a}{b}\in\QQ/\ZZ~|~p \text{ does not divide }b\}$ as
$$
V_p\left(\frac{a}{p^r-1}\right)=\frac{1}{r(p-1)}\cdot(\text{sum of the }p\text{-adic digits of }a)
$$
if $0\leq a\leq p^r-2$. This is well defined, since every class in $(\QQ/\ZZ)_{\text{not }p}$ has a representative in such form, and the computed value of $V_p$ does not depend on the chosen representative.

The Stickelberger theorem can be stated in terms of the $V$ function: if $\epsilon:k_r^\times\to\CC^\times$ is the Teichmüller character (which generates the cyclic group of multiplicative characters of $k_r^\times$) then the $q^r$-adic valuation of the Gauss sum
$$
G(\epsilon^a):=-\sum_{t\in k_r}\epsilon^a(t)\psi_r(t)
$$
for $0\leq a\leq q^r-2$ is given by $V_p\left(\frac{a}{q^r-1}\right)$. See \cite[Section 13]{K07} for details.

We first treat the case where $f(x)=x^d+\lambda x^e$ is a binomial, with $d>e>1$. In this case, we will see that the monodromy of $\FF_f$ is finite if and only if the monodromy of the sheaf on $\AAA^2_k$ whose Frobenius trace function is the two-parameter function
$$
(k_r;s,t)\mapsto\sum_{x\in k_r}\psi_r(x^d+sx^e+tx)
$$
is finite. This is one particular case of the multi-parameter families studied in \cite{KRLT20,KT24}, for which there are numerical criteria available. In \cite[Theorem 11.2.3]{KT24}, Katz and Tiep have classified all such multi-parameter families with finite monodromy, which therefore solves the problem in this case.

Next, we will consider the case of polynomials of Belyi type, that is, of the form $f(x)=(x-\alpha)^d(x-\beta)^e$ with $\alpha\neq\beta$, giving a numerical criterion for the finiteness of the monodromy of the corresponding local system (which does not actually depend on $\alpha$ and $\beta$). We then use this criterion, combined with the results in \cite{KT24}, to determine the complete set of pairs of exponents $(d,e)$ for which the corresponding local system has finite monodromy.

\section{Two-parameter families of exponential sums}

In this section we will collect some basic facts about the sheaves we are interested in. We fix a polynomial $f(x)=\sum_{i=0}^d a_ix^i$, for which we will assume the following condition holds:

\begin{center}
($\ast$) \emph{There exists some non-zero $i$ which is not a power of $p$ such that $a_i\neq 0$.}
\end{center}

Recall that two polynomials in $k[x]$ are said to be \emph{Artin-Schreier equivalent} if their difference is of the form $h^p-h$ for some $h\in k[x]$. Since the polynomials of this form a subgroup of $k[x]$, this defines an equivalence relation on $k[x]$. We will say that $g(x)=\sum b_ix^i\in k[x]$ is an \emph{Artin-Schreier reduced form} of $f$ if it is Artin-Schreier equivalent to $f$ and $b_i=0$ for all $i>0$ divisible by $p$. It is clear that every $f$ has an Artin-Schreier reduced form, which is unique modulo addition of a constant term.

For a positive integer $n$, let $n_{\hat p}$ denote its prime-to-$p$ part. 

\begin{lemma}
    There exists a dense open set $U_f\subseteq\Gmk$ such that the degree of the Artin-Schreier reduced form of $sf(x)$ is constant for $s\in U_f$, and equal to the maximum $e(f)$ of the $i_{\hat p}$ for the $i=1,\ldots,d$ such that $a_i\neq 0$.
\end{lemma}

\begin{proof}
    Let $I$ be the set of $i_{\hat p}$ for the $i\in\{1,\ldots,d\}$ such that $a_i\neq 0$, then $sf(x)=sa_0
+\sum_{i\in I}\sum_{j\geq 0}sa_{ip^j}x^{ip^j}$ is Artin-Schreier equivalent to $g_s(x)=sa_0+\sum_{i\in I}(\sum_{j\geq 0}s^{p^{-j}}a_{ip^j}^{p^{-j}})x^i$, which is in Artin-Schreier reduced form. The degree of $g_s(x)$ is the maximum of the $i\in I$ such that $\sum_{j\geq 0}s^{p^{-j}}a_{ip^j}^{p^{-j}}\neq 0$, so if $e$ is the maximum element of $I$ we can take $U_f$ to be the Zariski open set of $\Gmk$ defined by $\sum_{j\geq 0}s^{p^{-j}}a_{ep^j}^{p^{-j}}\neq 0$.
\end{proof}

Note that, if the degree of $f$ is prime to $p$, we can take $U_f$ to be the entire $\Gmk$.

Consider the lisse, rank 1 sheaf $\LL_{\psi(sf(x)+tx)}$ on $\Gmk\times\AAA^1_k\times\AAA^1_k$, pull-back by the map $(s,t,x)\mapsto sf(x)+tx$ of the Artin-Schreier sheaf $\LL_\psi$ on $\AAA^1_k$, and let $\pi_{12}:\Gmk\times\AAA^1_k\times\AAA^1_k\to\Gmk\times\AAA^1_k$ be the projection onto the first two factors. The object $K_f:=\R\pi_{12!}(\LL_{\psi(sf(x)+tx)})$ on $\Gmk\times\AAA^1_k$ has Frobenius trace function
$$
(k_r;s,t)\mapsto\sum_{x\in k_r}\psi_r(sf(x)+tx)
$$

\begin{prop}
    The restriction of $K_f$ to the open set $U_f\times\AAA^1_k$ is of the form $\GGG_f[-1]$, where $\GGG_f$ is a geometrically irreducible lisse sheaf of rank $e(f)-1$, pure of weight $1$.
\end{prop}

\begin{proof}
    The stalk of $K_f$ at a geometric point $(s,t)$ of $U_f$ is $\R\Gamma_c(\AAA^1_k,\LL_{\psi(sf(x)+tx)})$. Since $sf(x)$ is Artin-Schreier equivalent to $g_s(x)$, $\LL_{\psi(sf(x)+tx)}\cong\LL_{\psi(g_s(x)+tx)}$, and $\R\Gamma_c(\AAA^1_k,\LL_{\psi(g_s(x)+tx)})$ is known to be concentrated on degree 1, of dimension $e(f)-1$ and pure of weight 1, since $e(f)$ is prime to $p$ (see eg. \cite[3.5]{D77}).

    Since $\pi_{12}$ is affine, the sheaf $\GGG_f:=\mathcal H^1(K_f)=\R^1\pi_{12!}(\LL_{\psi(sf(x)+tx)})$ is of perverse origin by \cite[Corollary 6]{K03}. And then by \cite[Proposition 12]{K03}, it is lisse on the open set $U_f\times\AAA^1_k$, where the dimension of its stalks is constant.

    In order to show that $\GGG_f$ is geometrically irreducible, pick some $s_0\in U_f(k)$ (passing to a finite extension of $k$ if necessary), then it suffices to show that the restriction of $\GGG_f$ to the line $\{s_0\}\times\AAA^1_k$ is geometrically irreducible. But this is the (naive) Fourier transform of the rank 1 lisse sheaf $\LL_{\psi(s_0f)}\cong\LL_{\psi(g_{s_0})}$ on $\AAA^1_k$, so it is geometrically irreducible by \cite[Theorem 7.3.8]{K90}.
\end{proof}

If the polynomial $f(x)$ does not satisfy condition ($\ast$) (that is, if $f(x)=a_0+\sum_{i=0}^r a_{p^i}x^{p^i}$) then $sf(x)$ is Artin-Schreier equivalent to $sa_0+(\sum_{i=0}^r (sa_{p^i})^{p^{-i}})x$, and the corresponding exponential sum is zero on the Zariski open set defined by $t+\sum_{i=0}^r (sa_{p^i})^{p^{-i}}\neq 0$, so we take $\GGG_f=0$.

Pick a square root of $q$, then $\FF_f:=\GGG_f(1/2)=(1/\sqrt{q})^{deg}\otimes\GGG_f$ is pure of weight 0, and has Frobenius trace function
$$
(k_r;s,t)\mapsto -\frac{1}{q^{r/2}}\sum_{x\in k_r}\psi_r(sf(x)+tx).$$
The following result is simply \cite[Theorem 10.2.6]{KT24} when $d$ is prime to $p$, we generalize it here to $f$ of arbitrary degree.


\begin{prop}\label{criterion-algint}
The lisse sheaf $\FF_f$ on $U_f\times\AAA^1$ has finite (arithmetic or geometric) monodromy group if and only if
$$
S(k_r;s,t):=\frac{1}{q^{r/2}}\sum_{x\in k_r}\psi_r(sf(x)+tx)
$$
is an algebraic integer for every $r\geq 1$ and $(s,t)\in k_r^\times\times k_r^\times$.
\end{prop}

\begin{proof}
Since the inclusion $U_f\times\Gm\hookrightarrow U_f\times\AAA^1$ is surjective on the fundamental groups, $\FF_f$ has finite monodromy if and only if its restriction to $U_f\times\Gm$ does. By \cite[Proposition 2.1]{KRLT20}, $\FF_f$ has finite monodromy on $U_f\times\Gm$ if and only if $S(k_r;s,t)$ is an algebraic integer for every $r\geq 1$ and every $(s,t)\in U_f(k_r)\times k_r^\times$ which, in particular, is true if it is an algebraic integer for every $r\geq 1$ and $(s,t)\in k_r^\times\times k_r^\times$.

Suppose now that $\FF_f$ has finite monodromy. We need to show that $S(k_r;s,t)$ is an algebraic integer if $(s,t)\in Z_f(k_r)\times k_r^\times$, where $Z_f$ is the complement of $U_f$. Let $s\in Z_f(k_r)$, and $g_s(x)$ an Artin-Schreier reduced form of $sf(x)$. If $g_s(x)+tx$ is constant, then $S(k_r;s,t)=q^{r/2}$ times a root of unity, which is an algebraic integer. If $g_s(x)+tx$ is linear non-constant, then $S(k_r;s,t)=0$. Assume $g_s(x)$ has degree $c\geq 2$, then $\FF_f$ restricted to $\{s\}\times\Gm$ is of perverse origin and constant rank $c-1$, so it is lisse by \cite[Proposition 12]{K03} and has finite monodromy by \cite[Corollary 10]{K03}. Its trace function at $(s,t)$ is $-S(k_r;s,t)$, so by \cite[Proposition 2.1]{KRLT20} we conclude that $S(k_r;s,t)$ is an algebraic integer.
\end{proof}

We now prove a general fact about two-paremeter families with finite monodromy which will be useful later.

\begin{prop}\label{hiloterms}
 Let $f(x)=x^d+a_{d-1}x^{d-1}+\cdots+a_ex^e\in\Fp[x]$ be a monic polynomial of degree $d$ with $a_e\neq 0$ and no constant term (that is, $e>0$) such that $\FF_f$ has finite monodromy. Then $d$ and $e$ are FM-exponents.
\end{prop}

\begin{proof} 
 Suppose that $\FF_f$ has finite monodromy. Then so does its pull-back by the map $(s,t)\mapsto (s^d,t)$, which has trace function
 $$
 (k_r;s,t)\mapsto -\frac{1}{q^{r/2}}\sum_{x\in k_r}\psi_r(s^df(x)+tx)=
 $$
 $$
 =-\frac{1}{q^{r/2}}\sum_{x\in k_r}\psi_r((sx)^d+a_{d-1}s(sx)^{d-1}+\cdots+a_es^{d-e}(sx)^e+(t/s)(sx))=
 $$
 $$
 =-\frac{1}{q^{r/2}}\sum_{x\in k_r}\psi_r(x^d+a_{d-1}sx^{d-1}+\cdots+a_es^{d-e}x^e+(t/s)x)
$$
which, after applying the automorphism $(s,t)\mapsto(s,t/s)$ becomes the sheaf with trace function
$$
-\frac{1}{q^{r/2}}\sum_{x\in k_r}\psi_r(x^d+a_{d-1}sx^{d-1}+\cdots+a_es^{d-e}x^e+tx).
$$
This is the restriction to $U_f\times\AAA^1$ of the sheaf $\R^1\pi_!\GGG$ on $\AAA^2$, where $\pi:\AAA^1\times\AAA^2\to\AAA^2$ is the projection $(x,s,t)\mapsto (s,t)$ and $\GGG=\LL_{\psi(x^d+a_{d-1}sx^{d-1}+\cdots+a_es^{d-e}x^e+tx)}$. Since $\pi$ is affine and $\GGG$ is lisse (so $\GGG[3]$ is perverse), this sheaf is of perverse origin \cite[Corollary 6]{K03}. Therefore, its restriction to the line $s=0$ must have finite monodromy too by \cite[Corollary 10]{K03}, so $d$ is a FM-exponent.

If $e=1$ there is nothing to prove, so assume $e>1$. As above, the pull-back of $\FF_f$ by the map $(s,t)\mapsto (a_e^{-1}s^{-e},t)$ also has finite monodromy, and has trace function 
$$
 (k_r;s,t)\mapsto -\frac{1}{q^{r/2}}\sum_{x\in k_r}\psi_r(a_e^{-1}s^{-e}f(x)+tx)=
 $$
 $$
 =-\frac{1}{q^{r/2}}\sum_{x\in k_r}\psi_r(\frac{1}{a_e}s^{d-e}(x/s)^d+\frac{a_{d-1}}{a_e}s^{d-e-1}(x/s)^{d-1}+\cdots+(x/s)^e+st(x/s))=
 $$
 $$
 =-\frac{1}{q^{r/2}}\sum_{x\in k_r}\psi_r(\frac{1}{a_e}s^{d-e}x^d+\frac{a_{d-1}}{a_e}s^{d-e-1}x^{d-1}+\cdots+x^e+stx)
$$
which, after applying the automorphism $(s,t)\mapsto(s,st)$ becomes the sheaf with trace function
$$
-\frac{1}{q^{r/2}}\sum_{x\in k_r}\psi_r(\frac{1}{a_e}s^{d-e}x^d+\frac{a_{d-1}}{a_e}s^{d-e-1}x^{d-1}+\cdots+x^e+tx).
$$
This is the restriction to $U_f\times\AAA^1$ of the sheaf $\R^1\pi_!\GGG$ on $\AAA^2$, where $\pi:\AAA^1\times\AAA^2\to\AAA^2$ is the projection $(x,s,t)\mapsto (s,t)$ and $\GGG=\LL_{\psi(\frac{1}{a_e}s^{d-e}x^d+\frac{a_{d-1}}{a_e}s^{d-e-1}x^{d-1}+\cdots+x^e+tx)}$. Since $\pi$ is affine and $\GGG$ is lisse (so $\GGG[3]$ is perverse), this sheaf is of perverse origin \cite[Corollary 6]{K03}. Therefore, its restriction to the line $s=0$ must have finite monodromy too by \cite[Corollary 10]{K03}, so $e$ is a FM-exponent.
\end{proof}

\section{Criterion for finite monodromy when $f$ is a binomial}

In this section we consider the case where $f(x)= x^d+\lambda x^e$ with $d>e>1$ and $\lambda\in k^\times$.

\begin{thm}\label{equivalence} Let $f(x)=x^d+\lambda x^e\in k[x]$ with $\lambda\neq 0$ and $d>e>1$. Then the sheaf $\mathcal F_f$ on $U_f\times\AAA^1_k$ has finite monodromy if and only if the sheaf ${\mathcal G}_{d,e}$ on $\Gmk^3$ whose trace function is
$$
(k_r;u,s,t)\mapsto -\frac{1}{q^{r/2}}\sum_{x\in k_r}\psi_r(ux^d+sx^e+tx)
$$
does.
\end{thm}

\begin{proof} $\FF_f$ has finite monodromy if and only if its pull-back by the finite map $(s,t)\mapsto(s^d,t)$ does, which has trace function
$$
(k_r;s,t)\mapsto -\frac{1}{q^{r/2}}\sum_{x\in k_r}\psi_r(s^dx^d+s^d\lambda x^e+tx)=
-\frac{1}{q^{r/2}}\sum_{x\in k_r}\psi_r((sx)^d+s^{d-e}\lambda (sx)^e+tx)=
$$
$$
=-\frac{1}{q^{r/2}}\sum_{x\in k_r}\psi_r((ux)^d+s^{d-e}\lambda (ux)^e+tx)=-\frac{1}{q^{r/2}}\sum_{x\in k_r}\psi_r(u^dx^d+s^{d-e}u^e\lambda x^e+tx).
$$
This, in turn, is the trace function of the pull-back of ${\mathcal G}_{d,e}$ under the finite map $(u,s,t)\mapsto(u^d,s^{d-e}u^e\lambda,t)$, so it has finite monodromy if and only if ${\mathcal G}_{d,e}$ does.
\end{proof}

By \cite[Theorem 2.8]{KRLT20}, we get the following numerical criterion for finite monodromy of $\FF_f$ (which is also valid when $d$ or $e$ is a multiple of $p$):

\begin{cor}
    The sheaf $\FF_f$ has finite monodromy if and only if
    $$
    V(x)+V(y)+V(-dx-ey)\geq\frac{1}{2}
    $$
for every $x,y\in(\QQ/\ZZ)_{\mathrm{prime\,to\,}p}$, not both zero.
\end{cor}

The problem of determining finite monodromy of ${\mathcal G}_{d,e}$ has been completely solved by Katz and Tiep in \cite[Theorem 11.2.3]{KT24}, and we have:

\begin{thm}
 Let $f(x)=x^d+\lambda x^e$ with $d>e>1$. Then the sheaf $\FF_f$ has finite monodromy if and only if
 \begin{enumerate}
  \item $d_{\hat p}=e_{\hat p}=1$
  \item $d_{\hat p}=1$ and $e$ is an FM-exponent.
  \item $e_{\hat p}=1$ and $d$ is an FM-exponent.
  \item $d_{\hat p}=p^a+1$ and $e_{\hat p}=p^b+1$ for some $a,b\geq 0$ ($>0$ if $p=2$).
  \item $d_{\hat p}=\frac{p^{ab}+1}{p^a+1}$ and $e_{\hat p}=\frac{p^{ac}+1}{p^a+1}$ for some $a>0$ and odd $b,c>1$.
  \item $p>2$, $d_{\hat p}=\frac{p^a+1}{2}$ and $e_{\hat p}=\frac{p^b+1}{2}$ for some $a,b\geq 0$.
  \item $p=2$, $\{d_{\hat p},e_{\hat p}\}=\{13,3\}$.
  \item $p=3$, $\{d_{\hat p},e_{\hat p}\}=\{7,4\},\{7,2\},\{5,4\}$ or $\{5,2\}$.
  \item $p=5$, $\{d_{\hat p},e_{\hat p}\}=\{3,2\}$ or $\{7\}$.
 \end{enumerate}
\end{thm}

\begin{proof} If $d_{\hat p}=e_{\hat p}=1$ then both $d$ and $e$ are powers of $p$, so $\FF_f=0$ trivially has finite monodromy. If $d_{\hat p}\neq 1$ but $e_{\hat p}=1$, so $d=p^\alpha d_{\hat p}$ and $e=p^\beta$, then $sf(x)+tx$ is Artin-Schreier equivalent to $s^{p^{-\alpha}}x^{d_{\hat p}}+(t+(s\lambda)^{p^{-\beta}})x$, so our object is the pull-back by the map $(s,t)\mapsto(s^{p^{-\alpha}},t+(s\lambda)^{p^{-\beta}})$ of sheaf $\FF_{x^{d_{\hat p}}}$, and we conclude by Katz and Tiep's Theorem \ref{KT1}. A similar argument solves the case $d_{\hat p}=1,e_{\hat p}\neq 1$.

Asume both $d_{\hat p},e_{\hat p}$ are $\neq 1$. By Theorem \ref{equivalence}, $\FF_f$ has finite monodromy if and only if the local system on $\Gm^3$ whose trace function is
$$
(k_r;u, s,t)\mapsto -\frac{1}{q^{r/2}}\sum_{x\in k_r}\psi_r(ux^d+sx^e+tx)
$$
does. Since $ux^d+sx^e+tx$ is Artin-Schreier equivalent to $u^{p^{-\alpha}}x^{d_{\hat p}}+s^{p^{-\beta}}x^{e_{\hat p}}+tx$, this is the push-forward under the map $(u,s,t)\mapsto(u^{p^{\alpha}},s^{p^{\beta}},t)$ of the local system with trace function
$$
(k_r;u,s,t)\mapsto -\frac{1}{q^{r/2}}\sum_{x\in k_r}\psi_r(ux^{d_{\hat p}}+sx^{e_{\hat p}}+tx)
$$
where $d_{\hat p},e_{\hat p}>1$ are prime to $p$. If $d_{\hat p}=e_{\hat p}$, we conclude by Katz and Tiep's Theorem \ref{KT1}. Otherwise, we conclude by Katz and Tiep's Theorem \cite[Theorem 11.2.3]{KT24}.
\end{proof}

\section{Criterion for finite monodromy when $f$ is of Belyi type}

In this section we consider the case where $f(x)$ is of the form $(x-\alpha)^d(x-\beta)^e$ with $\alpha\neq\beta$ and $d,e\geq 1$, at least one of which is prime to $p$ (otherwise, we could replace $f$ by $f^{1/p}$). In fact, we will see that the finiteness of the monodromy group depends only on $d$ and $e$, not on $\alpha$ or $\beta$.

\begin{thm}\label{Belyicrit} Let $f(x)=(x-\alpha)^d(x-\beta)^e\in k[x]$ with $\alpha\neq\beta$ and $d,e\geq 1$. Then the two-parameter local system $\mathcal F_f$ on $U_f\times\AAA^1_k$ has finite monodromy if and only if
$$
V(x)+V(y)+V(y-(d+e)x)+V(ex-y)+V(-ex)\geq\frac{3}{2}
$$
for every $x,y\in(\QQ/\ZZ)_{\text{prime to }p}$ with $x,y\neq 0$, and
$$
V(x)+V(-(d+e)x)\geq\frac{1}{2}
$$
for every $x\in(\QQ/\ZZ)_{\text{prime to }p}$ with $x\neq 0$.
\end{thm}

\begin{proof} 
 We argue as in \cite[Theorem 2.7]{KRLT20}. By proposition \ref{criterion-algint}, $\FF_f$ has finite monodromy if and only if $\frac{1}{q^{r/2}}\sum_{x\in k_r}\psi_r(sf(x)+tx)$ is an algebraic integer for every $r\geq 1$ and $s,t\in k_r^\times$ or, equivalently, if and only if the $q^r$-adic valuation of $\sum_{x\in k_r}\psi_r(sf(x)+tx)$ is $\geq 1/2$ for every $r\geq 1$ and $s,t\in k_r^\times$. Since $\sum_{x\in k_r}\psi_r(sf(x+\alpha)+tx)=\sum_{x\in k_r}\psi_r(sf(x)+tx-t\alpha)=\psi_r(-t\alpha)\sum_{x\in k_r}\psi_r(sf(x)+tx)$ and $\psi_r(-t\alpha)$ is a root of unity we may assume, without loss of generality, that $\alpha=0$. Moreover, since $\sum_{x\in k_r}\psi_r(sf(x)+tx)=\sum_{x\in k_r}\psi_r(sf(\beta x)+t\beta x)$, the local system associated to the polynomial $f(x)$ is the pull-back via the isomorphism $(s,t)\mapsto (\beta^{d+e}s,\beta t)$ of the local system associated to the polynomial $\beta^{-d-e}f(\beta x)$, so we may assume also that $\beta=1$, so that $f(x)=x^d(x-1)^e$.
 
 We invoke the Mellin transform as in \cite{KRLT20} to deduce that $\mathcal F_f$ has finite monodromy if and only if for every $r\geq 1$ and every pair of (possibly trivial) multiplicative characters $\chi,\eta$ of $k_r^\times$, the sum
 $$
 S_r(\chi,\eta):=\sum_{s,t\in k_r^\times}\chi(s)\eta(t)\sum_{x\in k_r}\psi_r(sf(x)+tx)
 $$
 has $q^r$-adic valuation $\geq 1/2$.
 
 Now 
 $$
 \sum_{s,t\in k_r^\times}\chi(s)\eta(t)\sum_{x\in k_r}\psi_r(sf(x)+tx)=\sum_{s,t\in k_r^\times}\chi(s)\eta(t)\sum_{x\in k_r}\psi_r(sx^d(x-1)^e+tx)=
 $$
 $$
 =\sum_{x\in k_r}\left(\sum_{s\in k_r^\times}\chi(s)\psi_r(sx^d(x-1)^e)\right)\left(\sum_{t\in k_r^\times}\eta(t)\psi_r(tx)\right).
 $$
 
 If $\chi$ and $\eta$ are both trivial, the sum is
 $$
 S_r(\triv,\triv)=\sum_{x\in k_r\backslash\{0,1\}}1+(q^r-1)^2-(q^r-1)=q^r(q^r-2)
 $$
 which is a multiple of $q^r$. If $\chi$ is trivial but $\eta$ is not, then
 $$
 S_r(\triv,\eta)=\sum_{x\in k_r}\left(\sum_{s\in k_r^\times}\psi_r(sx^d(x-1)^e)\right)\left(\sum_{t\in k_r^\times}\eta(t)\psi_r(tx)\right)=
 $$
 $$
 =\sum_{x\in k_r\backslash\{0,1\}}(-\overline\eta(x)G(\eta))+(q^r-1) G(\eta)=q^rG(\eta)
 $$
is also a multiple of $q^r$, where $G(\eta)=\sum_{t\in k_r^\times}\eta(t)\psi_r(t)$ is the Gauss sum (which is an algebraic integer). If $\eta$ is trivial but $\chi$ is not, then
$$
S(\chi,\triv)=-G(\chi)\sum_{x\in k_r\backslash\{0,1\}}\overline\chi^d(x)\overline\chi^e(x-1)=-G(\chi)\overline\chi(-1)^e\cdot J(\overline\chi^{d},\overline\chi^e)
$$
where $J(-,-)$ denotes the Jacobi sum. Finally, if both $\chi$ and $\eta$ are nontrivial, we have
$$
S(\chi,\eta)=\sum_{x\in k_r}G(\chi)\overline\chi(x^d(x-1)^e)G(\eta)\overline\eta(x)=
$$
$$
=G(\chi)G(\eta)\sum_{x\in k_r}\overline\chi(x^d(x-1)^e)\overline\eta(x)
=G(\chi)G(\eta)\sum_{x\in k_r\backslash\{0,1\}}\overline\chi^d\overline\eta(x)\overline\chi^e(x-1)=
$$
$$
=G(\chi)G(\eta)\overline\chi(-1)^e\cdot J(\overline\chi^d\overline\eta,\overline\chi^e).
$$

If $\chi^{d+e}\eta$ is non-trivial, then $J(\overline\chi^d\overline\eta,\overline\chi^e)=G(\overline\chi^d\overline\eta)G(\overline\chi^e)/G(\overline\chi^{d+e}\overline\eta)$. If $\chi=\epsilon^{-(q^r-1)x}$ and $\eta=\epsilon^{-(q^r-1)y}$ where $\epsilon$ is the Teichmüller character and $x,y\in\frac{1}{q^r-1}\ZZ/\ZZ$, by the Stickelberger theorem we get
$$
\mathrm{ord}_{q^r}S(\chi,\eta)=V(x)+V(y)+V(-dx-y)+V(-ex)-V(-(d+e)x-y)=
$$
$$
=V(x)+V(y)+V(-dx-y)+V(-ex)+V((d+e)x+y)-1
$$
since $V(t)+V(-t)=1$ if $t\neq 0$. We obtain the inequality in the statement after applying the change of variable $y\mapsto y-(d+e)x$ (so the condition $(d+e)x+y\neq 0$ becomes $y\neq 0$).

If $\chi^{d+e}\eta$ is trivial (that is, if $(d+e)x+y=0$) but $\chi^e$ is not, then $J(\overline\chi^d\overline\eta,\overline\chi^e)=G(\overline\chi^d\overline\eta)G(\overline\chi^e)/q^r$, so we get
$$
\mathrm{ord}_{q^r}S(\chi,\eta)=V(x)+V(y)+V(-dx-y)+V(-ex)-1=V(x)+V(y)=V(x)+V(-(d+e)x).
$$

Finally, if both $\chi^e$ and $\chi^d\eta$ are trivial, then $J(\overline\chi^d\overline\eta,\overline\chi^e)=q^r-2$, so (note that $q^r\neq 2$, since there are non-trivial characters in $k_r^\times$):
$$
\mathrm{ord}_{q^r}J(\overline\chi^d\overline\eta,\overline\chi^e)=\left\{\begin{array}{ll}
                           \frac{1}{ar} & \text{if }p=2 \\
                           0 & \text{if }p\neq 2
                          \end{array}
\right.
$$

In the first case, we get $\mathrm{ord}_{q^r}S(\chi,\eta)=V(x)+V(y)+\frac{1}{ar}$ and the finite monodromy condition translates to
$$
V(x)+V(y)\geq \frac{1}{2}-\frac{1}{ar}.
$$
But this needs to hold for every $r\geq 1$, so if we consider the pulled-back characters $\chi\circ\mathrm{Nm}$ and $\eta\circ\mathrm{Nm}$ in the extension of $k$ of degree $nr$ for any $n\geq 1$, we would get the condition
$$
V(x)+V(y)\geq \frac{1}{2}-\frac{1}{anr}.
$$
Letting $n\to\infty$, we conclude that, for the monodromy to be finite, we need
$$
V(x)+V(y)\geq\frac{1}{2}
$$
for any $p$. To sum up, whenever $(d+e)x+y=0$, we need
$$
V(x)+V(y)=V(x)+V(-(d+e)x)\geq\frac{1}{2}.
$$
\end{proof}

Note that the second condition is the criterion for finite monodromy of the local system $\FF_{x^{d+e}}$ corresponding to the monomial $x^{d+e}$, by \cite[Section 5]{K18}. So, if the monodromy of $\FF_f$ is finite, then $d+e$ must be an FM-exponent as given in Theorem \ref{KT1}. We will generalize this in Proposition \ref{hiloterms}.

\section{Finite monodromy results for $f$ of Belyi type}

In this section we will determine the pairs $(d,e)$ such that $\FF_f$ has finite monodromy, where $f(x)=x^d(x-1)^e$. We will assume without loss of generality that $d$ and $e$ are not both multiples of $p$ (otherwise, we could write $f(x)=g(x)^p$ for some $p$ and use Artin-Schreier reduction).

\begin{prop}
 Let $f(x)=x^d(x-1)^e$ and $g(x)=x^e(x-1)^d$. Then $\FF_f$ has finite monodromy if and only if $\FF_g$ does.
\end{prop}

\begin{proof}
 By Proposition \ref{criterion-algint}, $\FF_f$ (respectively $\FF_g$) has finite monodromy if and only if $F(s,t):=\frac{1}{q^{r/2}}\sum_{x\in k_r}\psi_r(sf(x)+tx)$ (resp. $G(s,t):=\frac{1}{q^{r/2}}\sum_{x\in k_r}\psi_r(sg(x)+tx)$) is an algebraic integer for every $r\geq 1$ and $s,t\in k_r^\times$. But
 $$
 q^{r/2}F(s,t)=\sum_{x\in k_r}\psi_r(sf(x)+tx)=\sum_{x\in k_r}\psi_r(sx^d(x-1)^e+tx)=
 $$
 $$
 =\sum_{x\in k_r}\psi_r((-1)^esx^d(1-x)^e+tx)=
 $$
 $$
 =\sum_{x\in k_r}\psi_r((-1)^es(1-x)^dx^e+t(1-x))=\psi_r(t)\sum_{x\in k_r}\psi_r((-1)^{d+e}s(x-1)^dx^e-tx)=
 $$
 $$
 =\psi_r(t)\sum_{x\in k_r}\psi_r((-1)^{d+e}sg(x)-tx)=\psi_r(t)q^{r/2}G((-1)^{d+e}s,-t)
$$
so those two statements are equivalent, since $\psi_r(t)$ is a root of unity.
\end{proof}

Combining this with Proposition \ref{hiloterms} we conclude that, if $\FF_f$ has finite monodromy, then all of $d$, $e$ and $d+e$ are FM-exponents. We will start then by searching for all such pairs $(d,e)$. Since we are excluding the case where both $d$ and $e$ are multiples of $p$, we will assume that $d$ is prime to $p$.

\begin{rmk}\label{congruences} We note a few easy observations that will be useful in the proof of the theorem. If $p$ is prime then:
 \begin{itemize}
     \item $p^a\equiv 1$ mod $p-1$ for every $a\geq 0$.
     \item If $a=bc$ with $c\geq 1$ odd, then $\frac{p^a+1}{p^b+1}=p^{a-b}-p^{a-2b}+\cdots+p^{2b}-p^b+1\equiv 1$ mod $p-1$, and $\frac{p^a+1}{p^b+1}\equiv 1$ mod $p$.
     \item If $p\geq 3$, then $\frac{p^a+1}{2}\equiv 1$ (respectively $\frac{p+1}{2}$) mod $p-1$ if $a\geq 0$ is even (resp. odd), and $\frac{p^a+1}{2}\equiv \frac{p+1}{2}$ mod $p$ for every $a\geq 1$.
 \end{itemize}

 Consequently, all FM-exponents for $p$ are congruent to $1$, $2$ or $\frac{p+1}{2}$ modulo $p-1$.
\end{rmk}

\begin{lemma}\label{quotient}
Let $p$ be a prime, $m,n$ two non-negative integers and $a,b,c,d$ non-negative integers (assumed positive if $p=2$).
\begin{enumerate}
    \item If $p^m\frac{p^a+1}{p^b+1}=p^n\frac{p^c+1}{p^d+1}\in\mathbb Z_+$, then $m=n$ and either $a=b$ and $c=d$, or $a=c$ and $b=d$.
    \item $p^m\frac{p^a-1}{p^b+1}=p^n\frac{p^c-1}{p^d+1}\in\mathbb Z_+$, then $m=n$, $a=c$ and $b=d$.
    \item If $p^m\frac{p^a+1}{p^b+1}=p^n\frac{p^c-1}{p^d+1}\in\mathbb Z_+$, then $m=n$ and either $p=2$ and $a=b,c=2,d=1$ or $(a,b,c,d)=(3,1,4,2)$; or $p=3$ and $a=b,c=1,d=0$ or $(a,b,c,d)=(1,0,2,1)$.
    \item If $p^m(p^a+1)=p^n\frac{p^c+1}{p^d+1}\in\mathbb Z_+$, then $m=n$ and either $p=3$, $(a,c,d)=(0,1,0)$ or $p=2$, $(a,c,d)=(1,3,1)$.
    \item If $p^m(p^a+1)=p^n\frac{p^c-1}{p^d+1}\in\mathbb Z_+$, then $m=n$ and either $p=5$, $(a,c,d)=(0,1,0)$ or $p=3$, $(a,c,d)\in\{(0,2,1),(1,2,0)\}$ or $p=2$, $(a,c,d)\in\{(1,4,2),(2,4,1)\}$.
\end{enumerate}
\end{lemma}

\begin{proof}
In every case, $m=n$ follows by comparing the $p$-adic valuations on both sides.
\begin{enumerate}
\item Substracting 1, we get $p^b\frac{p^{a-b}-1}{p^b+1}=p^d\frac{p^{c-d}-1}{p^d+1}$. These are 0 if and only if $a=b$ and $c=d$. Otherwise, comparing the $p$-adic valuations we get $b=d$, and then $p^a+1=p^c+1$ so $a=c$.
\item Note that $a>b$ and $c>d$. Adding 1, we get $p^b\frac{p^{a-b}+1}{p^b+1}=p^d\frac{p^{c-d}+1}{p^d+1}$, so compating the $p$-adic valuations we have $b=d$, and then $p^a-1=p^c-1$ so $a=c$.
\item Modulo $p$ we get $\{1,\frac{1}{2}\}\equiv \{-1,-\frac{1}{2}\}$, so $p=2$ or $p=3$ and $bd=0$. Since $\frac{p^c-1}{p^d+1}\in\mathbb Z$, $c>d$. If $a=b$ then $p^c=p^d+2$, so either $p=2$ and $(c,d)=(2,1)$ or $p=3$ and $(c,d)=(1,0)$. Assume $a>b$. If $p=2$, then modulo 4 we get $2^b+2^d\equiv 2$, so either $b=1$ or $d=1$. If $b=1$, adding 1 we get $2^2\frac{2^{a-2}+1}{3}=2^d\frac{2^{c-d}+1}{2^d+1}$, so by (1) we get $d=2,a=3,c=4$. If $d=1$, substracting 1 we get $2^b\frac{2^{a-b}-1}{2^b+1}=2^2\frac{2^{c-2}-1}{3}$, which is not possible by (2).

If $p=3$ and $b=0$ then $3\frac{3^{a-1}+1}{2}=3^d\frac{3^{c-d}+1}{3^d+1}$, so $a=d=1,c=2$ by (1). If $d=0$ then $3^b\frac{3^{a-b}-1}{3^b+1}=3\frac{3^{c-1}-1}{2}$, which is not possible by (2).

\item Substracting 1 we get $p^a=p^d\frac{p^{c-d}-1}{p^d+1}$, so $a=d$ and $p^{c-d}=p^d+2$, which is only possible for $p=2,d=1,c=3$ or $p=3,d=0,c=1$.
\item We have $p^c=p^{a+d}+p^a+p^d+2$, and modulo $p-1$ this is $1\equiv 5$, so $p=2,3$ or $5$. If $p=5$ then the five summands must coincide, so $a=d=0$. If $p=3$ then either $a=0,d=1$ or $a=1,d=0$. If $p=2$ then either $a=1,d=2$ or $a=2,d=1$.
\end{enumerate}
\end{proof}

\begin{thm}\label{FMpairs}
 Let $A,B$ be positive integers such that $A$ is prime to $p$ and all of $A$, $B$, $A+B$ are FM-exponents. Then $(A,B)$ is one of the following pairs (or their reversed pairs):
 \begin{enumerate}
  \item $(\frac{p^a+1}{p^b+1},p^b\frac{p^a+1}{p^b+1})$ for $b\geq1$ and $a\geq b$ an odd multiple of $b$.
  \item $(\frac{p^{a+2b}+1}{p^b+1},p^b\frac{p^a+1}{p^b+1})$ for $b\geq1$ and $a\geq b$ an odd multiple of $b$.
  \item $p\geq 3$, $(\frac{p^a+1}{2},\frac{p^a+1}{2})$ for $a\geq 0$.
  \item $p=2$, $(3,10)$, $(5,6)$, $(5,8)$, $(5,17)$, $(5,52)$, $(9,2)$, $(9,4)$, $(9,11)$, $(9,13)$, $(9,17)$, $(9,34)$, $(9,43)$, $(9,48)$, $(11,2)$, $(11,13)$, $(11,57)$, $(13,44)$, $(13,228)$, $(17,26)$, $(17,40)$, $(33,10)$, $(33,24)$, $(33,172)$, $(33,208)$, $(65,176)$, $(171,34)$ or $(205,36)$.
  \item $p=2$, $(2^a+1,2^a+1)$ for $a\geq 1$.
  \item $p=2$, $(\frac{2^a+1}{2^b+1},\frac{2^a+1}{2^b+1})$ for $b\geq 1$, $a$ an odd multiple of $b$.
  \item $p=2$, $(1,2^a+1)$ for $a\geq 1$.
  \item $p=2$, $(2^a+1,2^a)$ for $a\geq 1$.
  \item $p=2$, $(3,2^a+1)$ for $a\geq 1$.
  \item $p=2$, $(2^a+1,3\cdot 2^a)$ for $a\geq 1$.
  \item $p=2$, $(2^a+1,\frac{2^{3a}+1}{2^a+1})$ for $a\geq1$.
  \item $p=2$, $(\frac{2^{3a}+1}{2^a+1},2^a(2^a+1))$ for $a\geq1$.
  \item $p=2$, $(2^a+1,\frac{2^a+1}{3})$ for odd $a\geq1$.
  \item $p=2$, $(1,\frac{2^a+1}{3})$ for odd $a\geq 1$.
  \item $p=2$, $(\frac{2^a+1}{3},2^a)$ for odd $a\geq1$.
  \item $p=2,$ $(3,\frac{2^{2a}+1}{5})$ for odd $a\geq 1$.
  \item $p=2$, $(\frac{2^{2a}+1}{5},3\cdot2^{2a})$ for odd $a\geq 1$.
  \item $p=2$, $(5,\frac{2^a+1}{3})$ for odd $a\geq1$.
  \item $p=2$, $(\frac{2^a+1}{3},5\cdot2^a)$ for odd $a\geq 1$.
  \item $p=2$, $(\frac{2^{3a}+1}{9},\frac{2^{3a}+1}{3})$ for odd $a\geq1$.
  \item $p=2$, $(\frac{2^{3a}+1}{9},2\frac{2^{3a}+1}{9})$ for odd $a\geq1$.
  \item $p=3$, $(2,5)$, $(4,3)$, $(4,10)$, $(5,7)$, $(10,63)$, $(28,45)$ or $(61,12)$.
  \item $p=3$, $(2,3^a+1)$ for $a\geq 0$.
  \item $p=3$, $(3^a+1,2\cdot 3^a)$ for $a\geq 0$.
  \item $p=3$, $(3^a+1,\frac{3^a+1}{2})$ for $a\geq 0$.
  \item $p=3$, $(1,\frac{3^a+1}{2})$ for $a\geq 0$.
  \item $p=3$, $(\frac{3^a+1}{2},3^a)$ for $a\geq 0$.
  \item $p=3$, $(4,\frac{3^a+1}{2})$ for $a\geq 0$.
  \item $p=3$, $(\frac{3^a+1}{2},4\cdot 3^a)$ for $a\geq 0$.
  \item $p=3$, $(\frac{3^a+1}{2},\frac{3^a+1}{4})$ for odd $a\geq 1$.
  \item $p=3$, $(\frac{3^a+1}{4},\frac{3^a+1}{4})$ for odd $a\geq 1$.
  \item $p=3$, $(2,\frac{3^a+1}{4})$ for odd $a\geq 1$.
  \item $p=3$, $(\frac{3^a+1}{4},2\cdot3^a)$ for odd $a\geq 1$.
  \item $p=5$, $(1,6)$, $(2,5)$, $(3,7)$, $(6,7)$ or $(6,15)$.
  \item $p=5$, $(2,\frac{5^a+1}{2})$ for $a\geq0$.
  \item $p=5$, $(\frac{5^a+1}{2},2\cdot 5^a)$ for $a\geq0$.
  \item $p=7$, $(2,2)$.
 \end{enumerate}
\end{thm}

\begin{proof}
 We study the possible cases depending on the types of $A$ and $B$ according to the classification of FM-exponents given in Theorem \ref{KT1}. For some $p$'s there is non-empty intersection between different types, so:
 \begin{itemize}
    \item  For $p=2$ we consider $3=2+1$ to be of type 1, even though it is also of type $3$ as $\frac{2^3+1}{2+1}$. Therefore, for exponents of type 3, $\frac{2^a+1}{2^b+1}$, we will assume $a\geq 5$ if $b=1$.
    \item For $p=3$ we consider $2=1+1$ to be of type 1, even though it is also of type 2 as $\frac{3+1}{2}$. Therefore, for exponents of type 2, $\frac{3^a+1}{2}$, we will assume $a\geq 2$.
 \end{itemize}

 \bigskip
{\bf Type (1,1):} Here $A=p^a+1$, $B=p^b(p^c+1)$, $A+B=1+p^a+p^b+p^{b+c}$. Depending of the type of $A+B$, we have:

\begin{itemize}
   \item $1+p^a+p^b+p^{b+c}=p^d(p^e+1)$, which modulo $p-1$ is $4\equiv 2$, so $p\leq 3$. If $p=2$ then $a,c,e>0$ and modulo 2 we get $1+2^b\equiv 2^d$, so either $b=0$, in which case $2+2^a+2^c=2^d+2^{d+e}$ so $a=1$ or $c=1$ (and we get the pairs $\mathbf{(3,2^c+1)_2}$ and their reversed pairs) or $a=c$ (and we get the pairs $\mathbf{(2^a+1,2^a+1)_2}$); or or $d=0$, and then $2^a+2^b+2^{b+c}=2^e$, so $a=b,c=1$ (and we get the pairs $\mathbf{(2^a+1,3\cdot 2^a)_2}$).
  
  If $p=3$ then three of the left hand side exponents must coincide, so either $a=b=0$ (and we get the pairs $\mathbf{(2,3^c+1)_3}$), $b=c=0$ (and we get the pairs $\mathbf{(3^a+1,2)_3}$) or $a=b,c=0$ (and we get the pairs $\mathbf{(3^a+1,2\cdot 3^a)_3}$).
  
  \item $p\geq 3$ and $1+p^a+p^b+p^{b+c}=p^d\frac{p^e+1}{2}$, so $2+2p^a+2p^b+2p^{b+c}=p^d+p^{d+e}$ and modulo $p-1$ we get $8\equiv 2$, so $p=3$ or $7$. If $p=3$ then we may assume $e>0$. If $d>0$ then the sum is a multiple of $3$, so either $a=b=0$ or or $b=c=0$ and we get pairs whose sum is of type 1. So $d=0$ and $1+2\cdot 3^a+2\cdot 3^b+2\cdot 3^{b+c}=3^e$, which implies $a=0$ (and then $b=1$, $c=1$ and we get the pair $\mathbf{(2,12)_3}$, part of \#29) or $b=0$ (and then either $a=1$, $c=2$ and we get the pair $\mathbf{(4,10)_3}$ or $a=2$, $c=1$ and we get the reversed pair $\mathbf{(10,4)_3}$). If $p=7$ then modulo 7 we get $1+7^a+7^b+7^{b+c}\equiv 4\cdot 7^d$, so $a=b=c=d=0$ and we get the pair $\mathbf{(2,2)_7}$.
  
  \item $1+p^a+p^b+p^{b+c}=p^d\frac{p^e+1}{p^f+1}$, which modulo $p-1$ is $4\equiv 1$, so $p=2$, and then $a,c>0$. By parity, exactly one of $b,d$ is non-zero. If $b=0$ then $1+2^{a-1}+2^{c-1}=2^{d-1}\frac{2^e+1}{2^f+1}$. We may assume $a,c>1$, as otherwise the sum is of type 1. Then $d=1$ by parity, and then $2^{a-1}+2^{c-1}=2^f\frac{2^{e-f}-1}{2^f+1}$. By lemma \ref{quotient} this is only possible for $f=1,a=2,c=4$ (and we get the pair $\mathbf{(5,17)_2}$), or $f=2,a=3,c=4$ (and we get the pair $\mathbf{(9,17)_2}$), or interchanging the values of $a$ and $c$ (which gives the reversed pairs).

  If $d=0$ then $2^a+2^b+2^{b+c}=2^f\frac{2^{e-f}-1}{2^f+1}$. Comparing the 2-adic valuations, either $a=b$, $a=f<b$ or $b=f<a$. In the first case, we may assume $c>1$ (otherwise the sum is of type 1). Then $2^{a+1}(1+2^{c-1})=2^f\frac{2^{e-f}-1}{2^f+1}$, so by lemma \ref{quotient} $f=c=2,a=b=1$ and we get the pair $\mathbf{(3,10)_2}$. If $a=f<b$ then $2^{b-a}(1+2^c)=2\frac{2^{e-f-1}-2^{f-1}-1}{2^f+1}$. If $f=1$ this is $2^2\frac{2^{e-3}-1}{3}$, so $b=a+2=3$, $3\cdot 2^c+4=2^{e-3}$, $c=2$ and we get the pair $\mathbf{(3,40)_2}$. If $f>1$ then $b=a+1$ and $2^c(2^f+1)=2(2^{e-f-2}-2^{f-2}-2^{f-1}-1)$. If $f=2$ this is $5\cdot 2^c=2^3(2^{e-6}-1)$, which does not have a solution. So $f>2$, and then $c=1$ and $2^{e-f-2}=2^f+2^{f-1}+2^{f-2}+2$, which implies $a=f=3,b=4$ and we get the pair $\mathbf{(9,48)_2}$.
  
  Finally, if $b=f<a$ then $2^{a-b}+2^c=2\frac{2^{e-f-1}-2^{f-1}-1}{2^f+1}$. If $f=1$ this is $3(2^{a-1}+2^c)=4(2^{e-3}-1)$, so either $c=a-1=1$ (and we get the pair $\mathbf{(5,6)_2}$), or $c=2<a-1$ and $2^{e-3}=4+3\cdot 2^{a-3}$, so $a=5$ (and we get the pair $\mathbf{(33,10)_2}$) or $a=3<c+1$ and $2^{e-3}=4+3\cdot 2^{c-2}$, so $c=4$ (and we get the pair $\mathbf{(9,34)_2}$). If $f>1$ then $2^{a-b-1}+2^{c-1}=\frac{2^{e-f-1}-2^{f-1}-1}{2^f+1}$ is odd, so either $a-b=1$ or $c=1$. If $a-b=1$ then $2^{c-1}(2^f+1)=2(2^{e-f-2}-2^{f-1}-2^{f-2}-1)$. If $f=2$ this is $5\cdot 2^{c-1}=2^3(2^{e-6}-1)$, which does not have a solution. So $f>2$ and then $c=2$ and $2^{e-f-2}=2^f+2^{f-1}+2^{f-2}+2$, which implies $f=b=3,a=4$ and gives the pair $\mathbf{(17,40)_2}$. If $c=1$ then $2^{a-b-1}(2^f+1)=2(2^{e-f-2}-2^{f-1}-2^{f-2}-1)$, which again does not have solution for $f=2$, so $f>2$, $a-b-1=1$ and $2^{e-f-2}=2^f+2^{f-1}+2^{f-2}+2$, which implies $f=b=3$, $a=5$ and gives the pair $\mathbf{(33,24)_2}$.

  \item $p=5$ and $1+5^a+5^b+5^{b+c}=7\cdot 5^d$, which is not possible modulo $4$.
\end{itemize}

{\bf Type (1,2):} $p\geq 3$, $A=p^a+1$, $B=p^b\frac{p^c+1}{2}$, $A+B=1+p^a+\frac{p^b+p^{b+c}}{2}$. Here $A\equiv 2$ and $B\equiv 1$ or $\frac{p+1}{2}$, so $A+B\equiv 3$ or $\frac{p+5}{2}$ modulo $p-1$. Depending of the type of $A+B$, we have:

\begin{itemize}
\item $1+p^a+\frac{p^b+p^{b+c}}{2}=p^d(p^e+1)$. This is $\equiv 2$ modulo $p-1$, which is only possible for $p=3$, and then $2+2\cdot 3^a+3^b+3^{b+c}=2\cdot 3^d+2\cdot 3^{d+e}$ where we may assume $c\geq 2$. If $d=0$ then $2\cdot 3^a+3^b+3^{b+c}=2\cdot 3^e$ so $a=b,c=1$, which contradicts $c\geq 2$. So $d>0$ and then modulo 3 we must have $a>0,b=0$, and $1+2\cdot 3^{a-1}+3^{c-1}=2\cdot 3^{d-1}+2\cdot 3^{d+e-1}$. Since $c>1$ this is not possible for $e>0$, so $e=0$ and $1+2\cdot 3^{a-1}+3^{c-1}=3^{d-1}+3^d$, which is only possible if $a=1$ (and then $c=3$ and we get the pair $\mathbf{(4,14)_3}$, part of \#28) or $a=c=1$, which contradicts the assumption that $c\geq 2$. 

\item $1+p^a+\frac{p^b+p^{b+c}}{2}=p^d\frac{p^e+1}{2}$. This is $\equiv 1$ or $\frac{p+1}{2}$ modulo $p-1$, which is only possible for $p=3$ or $5$. Modulo $p$ we get $2+2p^a+p^b=p^d$ so either $p=3,a=d=0$ or $b=0,d>0$ (and either $p=3$ or $a=0$). In the first case we get $3+3^b+3^{b+c}=3^e$, which is not possible if $c>0$. In the second case we have $3+2p^a+p^c=p^d+p^{d+e}$. If $p=3$ then either $a=1$ (and we get the pairs $\mathbf{(4,\frac{3^c+1}{2})_3}$) or $a=c$ (and we get the pairs $\mathbf{(3^a+1,\frac{3^a+1}{2})_3}$). If $p=5$ then $a=0$ and we get the pairs $\mathbf{(2,\frac{5^c+1}{2})_5}$.

\item $1+p^a+\frac{p^b+p^{b+c}}{2}=p^d\frac{p^e+1}{p^f+1}$. This is $\equiv 1$ modulo $p-1$, which is only possible for $p=3$ or $5$. Modulo $p$ we get either $b=0$ or $d=0$. If $d=0$ then $p^a+\frac{p^b+p^{b+c}}{2}=p^f\frac{p^{e-f}-1}{p^f+1}$ which has $p$-adic valuation $f$ and residue $-1$, which implies $a\geq b$. If $a=b$ then $p^a\frac{p^c+3}{2}=p^f\frac{p^{e-f}-1}{p^f+1}$. For $p=5$ this implies $a=f$ and $5+5^c+3\cdot 5^f+5^{c+f}=2\cdot 5^{e-f}$, only possible if $c=f=1$, which gives the pair $\mathbf{(6,15)_5}$. For $p=3$ the valuation of the LHS is $a+1$, so $f=a+1$ and $3^{a+c}+3^{a+1}+3^{c-1}+3=2\cdot 3^{e-f}$. We may assume $c>1$, and then this forces $a+1=c-1=1$ and we get the pair $\mathbf{(2,5)_3}$. If $a>b$ then the LHS has valuation $b$ and residue $\frac{1}{2}$, so $p=3$, $b=f>0$ and $2\cdot 3^{a-b}+3^c+2\cdot 3^a+3^b+3^{b+c}+3=2\cdot 3^{e-f}$ and, in particular, $e-f\geq 2$. Modulo 9 this is $2\cdot 3^{a-b}+3^b+3\equiv 2\cdot 3^{e-f}\equiv 0$, only possible if $a-b=1$, in which case $3^2+3^c+2\cdot 3^{b+1}+3^b+3^{b+c}=2\cdot 3^{e-f}$, and then $b=c=2$ and we get the pair $\mathbf{(28,45)_3}$.

If $d>0$ then $b=0$ and $\frac{3+2p^a+p^c}{2}=p^d\frac{p^e+1}{p^f+1}$. If $p=5$, modulo 5 we get $a=0$, and the LHS has 5-adic valuation $1$ so $d=1$. The 5-adic residue on the RHS is 1 which implies $c=1$ and we get the pair $\mathbf{(2,3)_5}$, part of \#35. If $p=3$ then $a>0$, and the 3-adic residue of $3+2\cdot 3^a+3^c$ must be $2$. Assuming $c>1$, this can only happen if $a=1$ and $c=2$, and we get the pair $\mathbf{(4,5)_3}$, part of \#28.

\item $p=5$ and $1+5^a+\frac{5^b+5^{b+c}}{2}=7\cdot 5^d$, so $2+2\cdot 5^a+5^b+5^{b+c}=14\cdot 5^d$. Modulo 5 this becomes $2+2\cdot 5^a+5^b=4\cdot 5^d$, which implies either $a=d=0$, so $5^b+5^{b+c}=10$, so $b=1,c=0$, which contradicts $c>0$; or $a=b=0$, and then $5+5^c=14\cdot 5^d$, which is not possible.
\end{itemize}

{\bf Type (1,3):} Here $A=p^a+1$, $B=p^b\frac{p^c+1}{p^d+1}$, $A+B=1+p^a+\frac{p^b+p^{b+c}}{p^d+1}$. Depending of the type of $A+B$, we have:

\begin{itemize}
\item $1+p^a+\frac{p^b+p^{b+c}}{p^d+1}=p^e(p^f+1)$. Modulo $p-1$ this is $3\equiv 2$, so it is only possible for $p=2$, and we may assume $a,f>0$. By parity, either $b=0$ or $e=0$. If $b=0$ then $2^{a-1}+\frac{2^{c-1}+2^{d-1}+1}{2^d+1}=2^{e-1}(1+2^f)$. If $c=d$ we get the pairs $\mathbf{(2^a+1,1)_2}$, so assume $c\geq 3d$. Let $d=1$ first (so $c\geq 5$), then $2^{a-1}+2\frac{2^{c-2}+1}{3}=2^{e-1}(1+2^f)$, so either $a=e=1$ (and then $3\cdot 2^{f-1}=2^{c-2}+1$, not possible for $c\geq 5$) or $a=2$ (and then $2^2\frac{2^{c-4}+1}{3}=2^{e-2}(1+2^f)$, so $e=4,f=1,c=7$ and we get the pair $\mathbf{(5,43)_2}$, part of \#18) or $e=2$ (and then $2^{a-3}+\frac{2^{c-3}-1}{3}=2^{f-1}$, so $a=3$ by parity and $2^{c-3}+2=2^f+ 2^{f-1}$, which is only possible for $c=3$ (which we exclude) or $c=5$, which gives the pair $\mathbf{(9,11)_2}$).

So assume $d>1$ from now on, then $a=1$ or $e=1$ by parity. In the first case we have $\frac{2^{c-2}+2^{d-2}+2^{d-1}+1}{2^d+1}=2^{e-2}(1+2^f)$. If $d=2$ this is $2^2\frac{2^{c-4}+1}{5}=2^{e-2}(1+2^f)$, so $2^{c-4}=2^2+2^f+2^{f+2}$, which is not possible. So $d>2$ and then $e=2$ by parity and $2^{d-2}\frac{2^{c-d}-1}{2^d+1}=2^f$, which is not possible for $d>1$ by lemma \ref{quotient}. If $e=1$ then $2^{a-1}+2^{d-1}\frac{2^{c-d}-1}{2^d+1}=2^f$, and then $a-1<f$, so comparing the 2-adic valuations we get $a=d$ and $2^{2a-1}\frac{2^{c-2d}+1}{2^d+1}=2^f$, so $c=3d$ and we get the pairs $\mathbf{(2^a+1,\frac{2^{3a}+1}{2^a+1})_2}$.

If $e=0$ then $2^a+2^b\frac{2^c+1}{2^d+1}=2^f$, so $a<f$ and comparing the $2$-adic valuations we get $a=b$ and $2^c+2^d+2=2^{f-a}+2^{f-a+d}$. Since we are excluding $d=1,c=3$ from type 3, this is only possible if $c=d$, which gives the pairs $\mathbf{(2^a+1,2^a)_2}$.

\item $p\geq 3$ and $1+p^a+\frac{p^b+p^{b+c}}{p^d+1}=p^e\frac{p^f+1}{2}$. Modulo $p-1$ this is $3\equiv 1$ or $\frac{p+1}{2}$, so $p=3$ or $5$. Modulo $p$ we have $2+2p^a+2p^b\equiv p^e$. For $p=5$ the only option is $a=b=e=0$, and then $5+3\cdot 5^d+2\cdot 5^c = 5^f+5^{d+f}$, so $c=d$ and we get the pair $\mathbf{(2,1)_5}$, part of \#35. For $p=3$, $f>1$ and exactly two of $a,b,e$ must be 0. If $a=e=0$ then $3^b\frac{3^c+1}{3^d+1}=3\frac{3^{f-1}-1}{2}$, so $b=1,d=c$ by lemma \ref{quotient} and we get the pair $\mathbf{(2,3)_3}$, part of \#27. If $b=e=0$ then $3^{a-1}+\frac{3^{c-1}-3^{d-1}}{3^d+1}=\frac{3^{f-1}-1}{2}$, which modulo 3 is $3^{a-1}+3^{c-1}-3^{d-1}\equiv 1$, so $a=1$ and $3^{d-1}\frac{3^{c-d}-1}{3^d+1}=3\frac{3^{f-2}-1}{2}$, so by lemma \ref{quotient} $c=d,f=2$ and we get the pair $\mathbf{(4,1)_3}$, part of \#28. Finally, if $a=b=0$ then $3+3^d\frac{3^{c-d}-1}{3^d+1}=3^e\frac{3^f+1}{2}$. If $d>1$ then modulo 9 we get $3\equiv 5\cdot 3^e$, which is not possible. So $d=1$ and $\frac{3^{c-1}+3}{4}=3^{e-1}\frac{3^f+1}{2}$, and then $3+3^{c-1}=2\cdot 3^{e-1}+2\cdot 3^{e+f-1}$, which is not possible if $f>0$.

\item $1+p^a+\frac{p^b+p^{b+c}}{p^d+1}=p^e\frac{p^f+1}{p^g+1}$. Modulo $p-1$ this is $3\equiv 1$, so $p=2$ or $3$. If $p=3$, modulo 3 we have $1+3^a+3^b=3^e$, so either $a=b=0$ or $e=0,a,b>0$. If $a=b=0$ then $\frac{2\cdot 3^{d-1}+3^{c-1}+1}{3^d+1}=3^{e-1}\frac{3^f+1}{3^g+1}$. If $d=1$ we get the pairs $\mathbf{(2,\frac{3^c+1}{4})_3}$, otherwise $e=1$ and $3^{d-1}\frac{3^{c-d}-1}{3^d+1}=3^g\frac{3^{f-g}-1}{3^g+1}$, impossible by lemma \ref{quotient}. If $e=0$ then $3^a+3^b\frac{3^c+1}{3^d+1}=3^g\frac{3^{f-g}-1}{3^g+1}$, where the RHS has 3-adic residue $-1$. Then we must have $a=b=g$, and $1+\frac{3^c+1}{3^d+1}=\frac{3^{f-g}-1}{3^g+1}$. If $c=d$ then $2\cdot 3^g+3=3^{f-g}$, so $g=1$ and we get the pair $\mathbf{(4,3)_3}$. Otherwise, $\frac{3^c+2\cdot 3^d+3}{3^d+1}=3^g\frac{3^{f-2g}+1}{3^g+1}$. If $d=1<c$ then the LHS is $3^2\frac{3^{c-2}+1}{4}$, so $g=2,c=3$ by lemma \ref{quotient} and we get the pair $\mathbf{(10,63)_3}$. If $c>d>1$ then $g=1$ and $3^{d-1}\frac{3^{c-d}-1}{3^d+1}=3\frac{3^{f-3}-1}{4}$, impossible by lemma \ref{quotient}.

If $p=2$ then by parity we either have $b=0$ or $e=0$. We consider the $b=0$ case first. If $c=d$ then we get the pairs $\mathbf{(2^a+1,1)_2}$ whose sum is of type 1, so assume $c\geq 3d$. Let $d=1$, then $c\geq 5$ and $2^a+2^2\frac{2^{c-2}+1}{3}=2^e\frac{2^f+1}{2^g+1}$. If $a=1$ then $e=1$ and $2\frac{2^{c-2}+1}{3}=2^g\frac{2^{f-g}-1}{2^g+1}$, not possible by lemma \ref{quotient} if $c>3$. If $a=2$ we get the pairs $\mathbf{(5,\frac{2^c+1}{3})_2}$. If $a>2$ then $e=2$ and $2^{a-2}+2\frac{2^{c-3}-1}{3}=2^g\frac{2^{f-g}-1}{2^g+1}$. If $a=3$ then $2^2\frac{2^{c-4}+1}{3}=2^g\frac{2^{f-g}-1}{2^g+1}$, so $c=7$ by lemma \ref{quotient} and we get the pair $\mathbf{(9,43)_2}$. If $a>3$ then $g=1$ and $2^{a-3}=2^{c-3}\frac{2^{f-c+2}-1}{3}$, so $a=c$ and we get the pairs $\mathbf{(2^a+1,\frac{2^a+1}{3})_2}$.

Still in the $b=0$ case, let now $d>1$. Then $2^a+2\cdot \frac{2^{c-1}+2^{d-1}+1}{2^d+1}=2^e\frac{2^f+1}{2^g+1}$. If $a=1$ then $2^2\frac{2^{c-2}+2^{d-2}+2^{d-1}+1}{2^d+1}=2^e\frac{2^f+1}{2^g+1}$. If $d=2$ we get the pairs $\mathbf{(3,\frac{2^c+1}{5})_2}$, so assume $d>2$. Then $e=2$ and $2^{d-2}\frac{2^{c-d}-1}{2^d+1}=2^g\frac{2^{f-g}-1}{2^g+1}$, which is impossible by lemma \ref{quotient}. If $a>1$ then $e=1$ and $2^{a-1}+2^{d-1}\frac{2^{c-d}-1}{2^d+1}=2^g\frac{2^{f-g}-1}{2^g+1}$, so two of $a-1,d-1,g$ must coincide. If $a=d$ then $2^{2d-1}\frac{2^{c-2d}+1}{2^d+1}=2^g\frac{2^{f-g}-1}{2^g+1}$, which is not possible by lemma \ref{quotient}. If $g=a-1<d-1$ then $2+2^{d-a}\frac{2^{c-d}-1}{2^d+1}=2^g\frac{2^{f-2g}+1}{2^g+1}$. If $d-a=1$ then $2^{d+1}\frac{2^{c-2d}+1}{2^d+1}=2^g\frac{2^{f-2g}+1}{2^g+1}$, so $g=d+1$, contradiction. Hence $g=1,a=2$ and $2^{d-3}\frac{2^{c-d}-1}{2^d+1}=2\frac{2^{f-3}-1}{3}$, which is not possible by lemma \ref{quotient}. Finally, if $g=d-1<a-1$ we have $2^{a-d}+2^d\frac{2^{c-2d}+1}{2^d+1}=2^g\frac{2^{f-2g}+1}{2^g+1}$. Since $g<d$ we must have $g=a-d$ and $2\frac{2^{c-2d}+1}{2^d+1}=2^g\frac{2^{f-3g}-1}{2^g+1}$, so $g=1,d=2,a=3,c=6$ by lemma \ref{quotient} and we get the pair $\mathbf{(9,13)_2}$.

We consider now the case $b>0$, then $2^a+2^b\frac{2^c+1}{2^d+1}=2^g\frac{2^{f-g}-1}{2^g+1}$, so two (and only two) of $a,b,g$ coincide. Suppose first that $a=b<g$. If $c=d$ we get the pairs $\mathbf{(2^a+1,2^a)_2}$ whose sum is of type 1, so assume $c\geq 3d$. Then $2^{a+1}\frac{2^{c-1}+2^{d-1}+1}{2^d+1}=2^g\frac{2^{f-g}-1}{2^g+1}$. If $d=1$ this is $2^{a+2}\frac{2^{c-2}+1}{2^d+1}=2^g\frac{2^{f-g}-1}{2^g+1}$, impossible by lemma \ref{quotient} since $a+2\geq 3$. So $d>1$ and then $g=a+1\geq 2$ and $\frac{2^{c-1}+2^{d-1}+1}{2^d+1}=\frac{2^{f-g}-1}{2^g+1}$, so $2\frac{2^{c-2}+2^{d-2}+2^{d-1}+1}{2^d+1}=2^g\frac{2^{f-2g}+1}{2^g+1}$. Since $g>1$ we must have $d=2$ and $2^3\frac{2^{c-4}+1}{2^2+1}=2^g\frac{2^{f-2g}+1}{2^g+1}$, and by lemma \ref{quotient} $g=3,c=6,a=b=2$ and we get the pair $\mathbf{(5,52)_2}$.

Suppose now that $a=g<b$. Then $2^{b-a}\frac{2^c+1}{2^d+1}=2\frac{2^{f-g-1}-2^{g-1}-1}{2^g+1}$. If $g=1$ then $2^{b-1}\frac{2^c+1}{2^d+1}=2^2\frac{2^{f-3}-1}{3}$. Then by lemma \ref{quotient} $b=3,c=d$ and we get the pair $\mathbf{(3,8)_2}$, part of \#15. If $g>1$ then $b=a+1$ and $2\frac{2^{c-1}+2^{d-1}+1}{2^d+1}=2^{g-1}\frac{2^{f-2g}+1}{2^g+1}$. If $d=1,c\geq 5$ then $2^2\frac{2^{c-2}+1}{3}=2^{g-1}\frac{2^{f-2g}+1}{2^g+1}$, impossible by lemma \ref{quotient}. So let $d>1$, then $g=a=2,b=3$ and $2^{d-1}\frac{2^{c-d}-1}{2^d+1}=2^g\frac{2^{f-3g}-1}{2^g+1}$. By lemma \ref{quotient} this implies $c=d$, which gives the pair $\mathbf{(5,8)_2}$.

Finally, suppose that $b=g<a$. Then $2^{a-b-1}+\frac{2^{c-1}+2^{d-1}+1}{2^d+1}=2^{g-1}\frac{2^{f-2g}+1}{2^g+1}$. Consider the case $d=1,c\geq 5$ first. Then $2^{a-b-1}+2\frac{2^{c-2}+1}{3}=2^{g-1}\frac{2^{f-2g}+1}{2^g+1}$. If $g=1$ then $a=b+1$ by parity, and $2^{c-1}+2^2=2^{f-2}$, not possible if $c>3$. If $g=2$ then $2^{a-b-2}+2\frac{2^{c-3}-1}{3}=2^2\frac{2^{f-6}-1}{2^2+1}$, and modulo 4 we get $a=b+3=5$ and $\frac{2^{c-4}+1}{3}=\frac{2^{f-6}-1}{2^2+1}$, and then by lemma \ref{quotient} we get $c=7$, which gives the pair $\mathbf{(33,172)_2}$. If $g>2$ then $a-b=2$ and $2^2\frac{2^{c-4}+1}{3}=2^{g-2}\frac{2^{f-2g}+1}{2^g+1}$, which by lemma \ref{quotient} implies $g=b=4,a=6,c=5$, which gives the pair $\mathbf{(65,176)_2}$. 

If $d>1$ then either $g=1$ or $a=b+1$ by parity. In the first case we have $2^{a-3}+2^{d-2}\frac{2^{c-d}-1}{2^d+1}=\frac{2^{f-3}-1}{3}$. If $c=d$ then $2^{a-3}+2^{a-2}+1=2^{f-3}$, so $a=3$ and we get the pair $\mathbf{(9,2)_2}$, assume $c\geq 3d$. Then by parity either $d=2$ or $a=3$. If $d=2$ then $2^{a-3}+\frac{2^{c-2}-1}{5}=\frac{2^{f-3}-1}{3}$, so $2^{a-3}+2^2\frac{2^{c-4}+1}{5}=2\frac{2^{f-4}+1}{3}$ and $a=4,c=6$ which gives the pair $\mathbf{(17,26)_2}$. If $a=3$ then $2^{d-2}\frac{2^{c-d}-1}{2^d+1}=2^2\frac{2^{f-5}-1}{3}$ which is not possible by lemma \ref{quotient}. Suppose now that $a=b+1$, then $1+\frac{2^{c-1}+2^{d-1}+1}{2^d+1}=2^{g-1}\frac{2^{f-2g}+1}{2^g+1}$. If $c=d$ then $g=b=2,a=3$ and we get the pair $\mathbf{(9,4)_2}$, assume $c\geq 3d$. Then $\frac{2^{c-2}+2^{d-2}+2^{d-1}+1}{2^d+1}=2^{g-2}\frac{2^{f-2g}+1}{2^g+1}$, and if $d=2$ then $2^2\frac{2^{c-4}+1}{2^d+1}=2^{g-2}\frac{2^{f-2g}+1}{2^g+1}$, so $b=g=4,a=5,c=6$ by lemma \ref{quotient} and we get the pair $\mathbf{(33,208)_2}$. If $d>2$ then $g=2$ by parity, and $2^{d-2}\frac{2^{c-d}-1}{2^d+1}=2^g\frac{2^{f-3g}-1}{2^g+1}$, which is impossible by lemma \ref{quotient}.
 
\item $p=5$ and $1+5^a+\frac{5^b+5^{b+c}}{5^d+1}=7\cdot 5^e$. Modulo 5 this becomes $1+5^a+5^b=2\cdot 5^e$, which implies $e=0$, and the only options are the pairs $\mathbf{(2,5)_5}$ and $\mathbf{(6,1)_5}$.
\end{itemize}

{\bf Type (1,4):} Here $p=5$, $A=5^a+1$, $B=5^b\cdot 7$, $A+B=1+5^a+7\cdot 5^b$. Modulo 4 we have $A\equiv 2$, $B\equiv 3$ so $A+B\equiv 1$ can only be of types 2 or 3. Depending of the type of $A+B$, we have:

\begin{itemize}
\item $1+5^a+7\cdot 5^b=5^c\frac{5^d+1}{2}$, which modulo 5 is $2+2\cdot 5^a+4\cdot 5^b\equiv 5^c$, so $b=c=0$ and then $15+2\cdot 5^a=5^d$, so $a=1$ and we get the pair $\mathbf{(6,7)_5}$.

\item $1+5^a+7\cdot 5^b=5^c\frac{5^d+1}{5^e+1}$, which modulo 5 is $1+5^a+2\cdot 5^b=5^c$, so $c=0$ and $a,b>0$. Then $5^a+7\cdot 5^b=5^e\frac{5^{d-e}-1}{5^e+1}$, which is impossible as both sides have different 5-adic residues.
\end{itemize}

{\bf Type (2,1):} Here $p\geq 3$, $A=\frac{p^a+1}{2}$, $B=p^b(p^c+1)$, $A+B=\frac{1+p^a}{2}+p^b+p^{b+c}$. We may assume $b>0$, since otherwise we get a reversed pair of type (1,2), and then $A+B$ is prime to $p$. Depending of the type of $A+B$, we have:

\begin{itemize}
\item $\frac{1+p^a}{2}+p^b+p^{b+c}=p^d+1$. Modulo $p$ we get $\frac{p+1}{2}\equiv 1$ or $2$, so $p=3$ and $d=0$, which is impossible as the sum is larger than 2.

\item $\frac{1+p^a}{2}+p^b+p^{b+c}=\frac{p^d+1}{2}$, so $p^a+2p^b+2p^{b+c}=p^d$, which modulo $p-1$ can only hold for $p=3,5$. For $p=5$ all exponents on the LHS must coincide and we get $a=b,c=0$ which gives the pairs $\mathbf{(\frac{5^a+1}{2},2\cdot 5^a)_5}$. If $p=3$ and $c=0$ we get $3^a+3^b+3^{b+1}$ which can never be a power of 3. So $c>0$ and then the exponents $a$ and $b$ must coincide, and $c=1$: we get the pairs $\mathbf{(\frac{3^a+1}{2},4\cdot 3^a)_3}$.

\item $\frac{1+p^a}{2}+p^b+p^{b+c}=\frac{p^d+1}{p^e+1}$. Modulo $p$ we get $\frac{1+p^a}{2}\equiv 1$, which is not possible if $a>0$.

\item $p=5$ and $\frac{1+5^a}{2}+5^b+5^{b+c}=7$, impossible as the LHS is $>7$.
\end{itemize}

{\bf Type (2,2):} Here $p\geq 3$, $A=\frac{p^a+1}{2}$, $B=p^b\frac{p^c+1}{2}$, $A+B=\frac{1+p^a+p^b+p^{b+c}}{2}$. Depending of the type of $A+B$, we have:

\begin{itemize}
\item $\frac{1+p^a+p^b+p^{b+c}}{2}=p^d(p^e+1)$, then $1+p^a+p^b+p^{b+c}=2p^d+2p^{d+e}$. If $p=3$ and $e=0$ then this is $3^d+3^{d+1}$ and three of the LHS exponents must coincide. Otherwise two pairs of exponents must coincide. Since $a,c>0$ the only option is $b=0,a=c$, and we get the pairs $\mathbf{(\frac{p^a+1}{2},\frac{p^a+1}{2})_p}$.

\item $\frac{1+p^a+p^b+p^{b+c}}{2}=p^d\frac{p^e+1}{2}$. Here we get the halved pairs of type (1,1) whose sum is of type 1, that is, $\mathbf{(\frac{3^a+1}{2},3^a)_3}$ for $p=3$.

\item $\frac{1+p^a+p^b+p^{b+c}}{2}=p^d\frac{p^e+1}{p^f+1}$. Modulo $p-1$ the LHS is $2$ or $\frac{p+3}{2}$ and the RHS is $1$, so $p=3$ and we may assume $a,c>1$. Modulo $3$ we get $2+2\cdot 3^b=3^d$, so $b=d=0$ and then $3^a+3^c+2\cdot 3^f+3^{a+f}+3^{c+f}=2\cdot 3^e$. So $3^a+3^c$ must be divisible by $3^f$, which implies $a,c\geq f$, and then $3^{a-f}+3^{c-f}+2+3^a+3^c=2\cdot 3^{e-f}$, so $e>f$ and the LHS is a multiple of $3$, which implies $a=f$ or $c=f$. In the first case we get $3+3^{c-a}+3^a+3^c=2\cdot 3^{e-a}$, only possible if $a=1$, which we explicitly exclude. Similarly if $c=f$, by symmetry.

\item $p=5$ and $\frac{1+5^a+5^b+5^{b+c}}{2}=7\cdot 5^d$, so $1+5^a+5^b+5^{b+c}=14\cdot 5^d$, which is not possible modulo 4.
\end{itemize}

{\bf Type (2,3):} Here $A=\frac{p^a+1}{2}$, $B=p^b\frac{p^c+1}{p^d+1}$, $A+B=\frac{1+p^a}{2}+\frac{p^b+p^{b+c}}{p^d+1}$. Depending of the type of $A+B$, we have:

\begin{itemize}
\item $\frac{1+p^a}{2}+\frac{p^b+p^{b+c}}{p^d+1}=p^e(p^f+1)$, then modulo $p$ either $b=0$ or $e=0$. If $b=0$ then $3+p^a+p^d+p^{a+d}+2p^c=2p^e+2p^{e+d}+2p^{e+f}+2p^{d+e+f}$, which modulo $p$ is $3\equiv 2p^e+2p^{e+f}$, only possible for $p=3$ and $e>0$, in which case we may assume $a>1$. Then modulo 9 we get $3+3^d+2\cdot 3^c\equiv 2\cdot 3^e+2\cdot 3^{e+f}$. Since the LHS is not a multiple of 9, we must have $e=1$, and $3^d+2\cdot 3^c\equiv 3+2\cdot 3^{f+1}$ modulo 9, which implies either $f=0$ (and then the sum of the pair is 6, which leaves $\mathbf{(5,1)_3}$ as the only option, part of \#26) or $d=1$ (and we get $3^a+3^{a+1}+2\cdot 3^c=2\cdot 3^2+2\cdot 3^{f+1}+2\cdot 3^{f+2}$, so $f=1,a=2,c=3$ which gives the pair $\mathbf{(5,7)_3}$).

If $b>0$ (so $e=0$) then modulo $p$ we get $\frac{1}{2}\equiv 1+p^f$, only possible if $p=3$ and $f=0$, in which case the sum of the pair is 2, which is not possible if $a>1$.

\item $\frac{1+p^a}{2}+\frac{p^b+p^{b+c}}{p^d+1}=p^e\frac{p^f+1}{2}$. Modulo $p-1$ we deduce that $p=3$ (so $a,f>1$), and modulo $3$ either $b=0$ or $e=0$. If $b=0$ then $3+3^a+3^d+3^{a+d}+2\cdot 3^c=3^e(1+3^f)(1+3^d)$. Then modulo 9 we get $3+3^d+2\cdot 3^c\equiv 3^e+3^{e+d}$, and since the LHS is not a multiple of 9, we must have $e=1$ and $3^a+3^d+3^{a+d}+2\cdot 3^c=3^{d+1}+3^{f+1}+3^{d+f+1}$. If $c=d$ we get the pairs $\mathbf{(\frac{3^a+1}{2},1)_3}$. Otherwise, $c>d$ and modulo $3^{d+1}$ we get $3^a+3^d\equiv 3^{f+1}$, which implies $d=f+1$ and $3^a+3^{a+d}+2\cdot 3^c=3^{d+1}+3^{2d}$. Then either $a=c$ and $3^{a+1}+3^{a+d}=3^{d+1}+3^{2d}$ (which implies $d=a=c$, contradiction) or $c=a+d$ and $3^a+3^{a+d+1}=3^{d+1}+3^{2d}$, which is not possible.

If $b>0$ (so $e=0$) then $3^a+3^{a+d}+2\cdot 3^b+2\cdot 3^{b+c}=3^f+3^{d+f}$ which has 3-adic valuation $f$. Since $v_3(3^a+3^{a+d})=a$, $v_3(2\cdot 3^b+2\cdot 3^{b+c})=b$ and $a<f$, we must have $a=b$ for the equality to hold, and then $3^{a+1}+3^{a+d}+2\cdot 3^{a+c}=3^f+3^{d+f}$. So $c=d$ and we get the pairs $\mathbf{(\frac{3^a+1}{2},3^a)_3}$. 

\item $\frac{1+p^a}{2}+\frac{p^b+p^{b+c}}{p^d+1}=p^e\frac{p^f+1}{p^g+1}$. Modulo $p-1$ we get $1$ or $\frac{p+1}{2}\equiv 0$, so $p=3$ and $a>1$. Modulo 3 we get $2+3^b\equiv 3^e$ so $b=0,e>0$. If $c=d$ we get the pairs $(\frac{3^a+1}{2},1)$ whose sum is of type 2, so assume $c\geq 3d$. Then modulo 9 we get $3+3^d\equiv 2\cdot 3^e$, so $d=e=1$ and $\frac{1}{4}(3+2\cdot 3^a+3^c)=3\frac{3^f+1}{3^g+1}$ (and then $f>g$, otherwise the sum of the pair would be 3, too small). Dividing by 3 we get $(1+3^g)(1+2\cdot 3^{a-1}+3^{c-1})=4(1+3^f)$, which modulo 9 is $(1+3^g)(1+2\cdot 3^{a-1})\equiv 4$, only possible if $g=1$, so $2\cdot 3^{a-1}+3^{c-1}=3^f$, and $a=c$, which gives the pairs $\mathbf{(\frac{3^a+1}{2},\frac{3^a+1}{4})_3}$.

\item $p=5$ and $\frac{1+5^a}{2}+\frac{5^b+5^{b+c}}{5^d+1}=7\cdot 5^e$. Modulo 5 this is $3+5^b\equiv 2\cdot 5^e$, which is not possible.
\end{itemize}

{\bf Type (2,4):} Here $p=5$, $A=\frac{5^a+1}{2}$, $B=5^b\cdot 7$, $A+B=\frac{1+5^a+14\cdot 5^b}{2}$. Modulo 4 we have $A\equiv 1$ or $3$, $B\equiv 3$ so $A+B\equiv 0$ or $2$ can only be of type 1: $\frac{1+5^a+14\cdot 5^b}{2}=5^c(5^d+1)$. Modulo 5 this is $1+4\cdot 5^b\equiv 2\cdot 5^c+2\cdot 5^{c+d}$, so $b=0$, $c>0$ and $15+5^a=2\cdot 5^c+2\cdot 5^{c+d}$, which implies $a=c=1$ and we get the pair $\mathbf{(3,7)_5}$.

{\bf Type (3,1):} Here $A=\frac{p^a+1}{p^b+1}$, $B=p^c(p^d+1)$, $A+B=\frac{p^a+1}{p^b+1}+p^c+p^{c+d}$. We may assume $c>0$, since otherwise we get a reversed pair of type (1,3), and then $A+B$ is prime to $p$. Depending of the type of $A+B$, we have:

\begin{itemize}
\item $\frac{p^a+1}{p^b+1}+p^c+p^{c+d}=p^e+1$. Modulo $p-1$ this is $3\equiv 2$, so $p=2$, and then $d,e>0$. If $a=b$ then $p^c+p^{c+d}=p^e$, which is not possible, so $a\geq 3b$, and $2^b\frac{2^{a-b}-1}{2^b+1}+2^c(2^d+1)=2^e$. Since $b,c<e$, comparing the 2-adic valuations we get $b=c$ and $2^b\frac{2^{a-2b}+1}{2^b+1}+2^d=2^{e-d}$, and by the same argument $b=d$ and $2^{a-2b}+2^b+2=2^{e-2b}+2^{e-b}$. Therefore, either $b=1$, and then $a=3$ or $a=5$ (and we get the pairs $\mathbf{(3,6)_2}$, part of the family \#1, and $\mathbf{(11,6)_2}$, part of the family \#2), or $a=3b$ (and we get the pairs $\mathbf{(\frac{2^{3b}+1}{2^b+1},2^b(2^b+1))_2}$).

\item $p\geq 3$ and $\frac{p^a+1}{p^b+1}+p^c+p^{c+d}=\frac{p^e+1}{2}$. Modulo $p$ we get $1\equiv \frac{1}{2}$, which is not possible.

\item $\frac{p^a+1}{p^b+1}+p^c+p^{c+d}=\frac{p^e+1}{p^f+1}$, modulo $p-1$ this is $3\equiv 1$, so $p=2$ or $3$. Substracting 1, we get $p^b\frac{p^{a-b}-1}{p^b+1}+p^c(p^d+1)=p^f\frac{p^{e-f}-1}{p^f+1}$. If $a=b$ then $c=f$ and $p^d+1=\frac{p^{e-f}-1}{p^f+1}$. If $p=3$, by lemma \ref{quotient} $d=0,f=1$ and we get the pair $\mathbf{(1,6)_3}$, part of the family \#33, and if $p=2$ (so $d>0$) $2^{d-1}=\frac{2^{e-f-1}-2^{f-1}-1}{2^f+1}$. If $d=1$ then $2^{e-f-1}=2^f+2^{f-1}+2$, then $f=2$ and we get the pair $\mathbf{(1,12)_3}$, part of the family \#17. Otherwise $\frac{2^{e-f-1}-2^{f-1}-1}{2^f+1}$ must be even, so $f=1$, and then $2^d+2^{d-1}+2=2^{e-2}$, so $d=2$ and we get the pair $\mathbf{(1,10)_2}$, part of the family \#19.

Assume from now on that $a\geq 3b$. Comparing the $p$-adic valuations, two of $b,c,f$ must coincide. Suppose $b=c$ first. Then $p^b\frac{p^{a-2b}+1}{p^b+1}+p^d=p^{f-b}\frac{p^{e-f}-1}{p^f+1}$. If $p=3$, the 3-adic residues of the terms on the LHS are 1, and on the RHS it is 2, so $b=d=f-b$, and then $\frac{3^{a-2b}+3^b+2}{3^b+1}=\frac{3^{e-f}-1}{3^f+1}$. Since $f=2b\geq 2$, if $b>1$ the equality is impossible modulo 9. So $b=1,f=2$ and $5\cdot 3^{a-2}+3^3=2\cdot 3^{e-2}$, which implies $a=5$ and we get the pair $\mathbf{(61,12)_3}$. If $p=2$, two of $b,d,f-b$ must coincide. If $b=d<f-b$ then $\frac{2^{a-2b}+2^b+2}{2^b+1}=2^{f-2b}\frac{2^{e-f}-1}{2^f+1}$. If $b=1$ then $2^2\frac{2^{a-4}+1}{3}=2^{f-2}\frac{2^{e-f}-1}{2^f+1}$, which is impossible by lemma \ref{quotient}. So $b>1$, then $f-2b=1$ and $\frac{2^{a-2b-1}+2^{b-1}+1}{2^b+1}=\frac{2^{e-f}-1}{2^f+1}$, which modulo 4 is $2^{a-2b-1}+2^{b-1}\equiv 2$, so $b=2,f=5$ and then modulo 16 we get $2^{a-5}+3\equiv -5$, so $a=8$, which is not an odd multiple of $b$. If $b=f-b<d$ then $\frac{2^{a-2b}+1}{2^b+1}+2^{d-b}=\frac{2^{e-f}-1}{2^f+1}$, which modulo 4 is $1+2^{d-b}\equiv -1$ (if $a=3b$) or $1-2^b+2^{d-b}\equiv -1$ (if $a>3b$), so either $b=1$ or $d=b+1$. In the first case, $f=2$, we may assume $a\geq 5$ and modulo 16 we get $2^{a-2}+2^{d-1}\equiv 8$, so either $a=5$ (and then $2^4+2^{d-1}+2^{d+1}=2^{e-2}$, which is not possible) or $d=4$ (and then $2^{a-2}+2^a+2^7=2^{e-2}+2^{e-1}$, $a=9$ and we get the pair $\mathbf{(171,34)_2}$). In the second case, we assume $b>1$ (so $f=2b\geq 4$), and then modulo 8 we get $3-2^b\equiv -1$, so $b=2,f=4$, $2^a+2^{a-4}+2^7+2^6=2^{e-4}+2^{e-2}$, and then $a=10$ and we get the pair $\mathbf{(205,36)_2}$. Finally, if $d=f-b<b$ then $2^{b-d}\frac{2^{a-2b}+1}{2^b+1}+1=\frac{2^{e-f}-1}{2^f+1}$ and $f\geq 2$, so modulo 4 we get $b-d=1$, and then modulo 8 we get $3\equiv -1$ since $f=2d+1\geq 3$, contradiction.

Suppose $b=f$ now. Then $p^{c-b}(p^d+1)=\frac{p^{e-b}-p^{a-b}}{p^b+1}=p^{a-b}\frac{p^{e-a}-1}{p^b+1}$, so $a=c$ and, by lemma \ref{quotient}, if $p=3$ this implies $d=0,b=1$ and we get the pairs $\mathbf{(\frac{3^a+1}{4},2\cdot 3^a)_3}$, and if $p=2$ then either $b=1,d=2$ and we get the pairs $\mathbf{(\frac{2^a+1}{3},5\cdot 2^a)_2}$ or $b=2,d=1$ and we get the pairs $\mathbf{(\frac{2^a+1}{5},3\cdot 2^a)_2}$.

Suppose now that $c=f$, then $p^{b-c}\frac{p^{a-b}-1}{p^b+1}+p^d+1=\frac{p^{e-f}-1}{p^f+1}$. If $p=3$, modulo 3 we get $3^d+2\equiv 3^{b-c}$, so $d=0,b>c$, and modulo 9 we get $3^f+3^{b-c}\equiv 3$, so either $f=1$ of $b-c=1$. In the first case, $3^{b-c}\frac{3^{a-b}-1}{3^b+1}=3^2\frac{3^{e-3}-1}{4}$, so $b-c=2$ and $b=1$ by lemma \ref{quotient}, which is impossible. In the second case, $3^b\frac{3^{a+2b}+1}{3^b+1}=3^{f-1}\frac{3^{e-2f}+1}{3^f+1}$, which is also impossible by lemma \ref{quotient}.

If $p=2$, we assume $d\geq 1$ so $b>c$ by parity. Modulo 4 we get $2^d+2\equiv 2^f+2^{b-c}$, so either all three or exactly one of $d,f,b-c$ is $1$. If $d=f=b-c=1$ then $b=2$ and $2^3\frac{2^{a-4}+1}{5}=2^2\frac{2^{e-3}-1}{3}$, impossible by lemma \ref{quotient}. If $d=1,f,b-c>1$, modulo 8 we get $2^f+2^{b-c}\equiv 4$, so either $f=2$ (and then $2^{b-c}\frac{2^{a-b}-1}{2^b+1}=2^4\frac{2^{e-6}-1}{5}$, impossible by lemma \ref{quotient}) or $b-c=2$ (and then $2^{b+2}\frac{2^{a-2b}+1}{2^b+1}=2^f\frac{2^{e-2f}+1}{2^f+1}$, so $f=c=b+2$ by lemma \ref{quotient}, contradiction). If $f=1,d,b-c>1$, we may assume $e\geq 5$, and modulo 8 we get $2^d+4\equiv 2^{b-c}$, so either $d=2$ (and then $2^{b-c}\frac{2^{a-b}-1}{2^b+1}=2^4\frac{2^{e-5}-1}{3}$, impossible by lemma \ref{quotient}) or $b-c=2$ (and then $b=3$ and $\frac{2^{a-3}-1}{9}+2^{d-2}=\frac{2^{e-3}-1}{3}$, which modulo 4 gives $2^{d-2}-1\equiv 1$, so $d=3$ and $2^3\frac{2^{a-6}+1}{9}=2^2\frac{2^{e-5}-1}{3}$, which is impossible). If $b-c=1,d,f>1$ then $2^{b+1}\frac{2^{a-2b}+1}{2^b+1}+2^d=2^f\frac{2^{e-2f}+1}{2^f+1}$. Since $b+1=f+2>f$, this implies $d=f$ and $2^{b+1}\frac{2^{a-2b}+1}{2^b+1}=2^{2f}\frac{2^{e-3f}-1}{2^f+1}$, so $f=2,b=3$ and the equality is impossible by lemma \ref{quotient}.

\item $p=5$ and $\frac{5^a+1}{5^b+1}+5^c+5^{c+d}=7$, impossible as the LHS is $>7$.
\end{itemize}

{\bf Type (3,2):} Here $p\geq 3$, $A=\frac{p^a+1}{p^b+1}$, $B=p^c\frac{p^d+1}{2}$, $A+B=\frac{p^a+1}{p^b+1}+\frac{p^c+p^{c+d}}{2}$. We may assume $c>0$, since otherwise we get a reversed pair of type (2,3), and then $A+B$ is prime to $p$. Depending of the type of $A+B$, we have:

\begin{itemize}
\item $\frac{p^a+1}{p^b+1}+\frac{p^c+p^{c+d}}{2}=p^e+1$. If $a=b$ then $p^c+p^{c+d}=2p^e$, which is not possible if $d>0$. So $a>b$ and $p^b\frac{p^{a-b}-1}{p^b+1}+p^c\frac{1+p^d}{2}=p^e$. Since the $p$-adic valuations of the summand on the LHS are $b$ and $c$ and $c<e$, we must have $b=c$ and $\frac{p^{a-b}-1}{p^b+1}+\frac{1+p^d}{2}=p^{e-c}$, which modulo $p$ is $-\frac{1}{2}\equiv 0$, so not possible.

\item $\frac{p^a+1}{p^b+1}+\frac{p^c+p^{c+d}}{2}=\frac{p^e+1}{2}$. Modulo $p$ this is $1\equiv \frac{1}{2}$, which is not possible.

\item $\frac{p^a+1}{p^b+1}+\frac{p^c+p^{c+d}}{2}=\frac{p^e+1}{p^f+1}$, modulo $p-1$ the LHS is $\equiv 2$ or $\frac{p+3}{2}$ and the RHS is $\equiv 1$, so $p=3$ and we may assume $d>1$. Note that $e\geq 3f$, otherwise the sum would be 1. If $a=b$ then $3^c\frac{3^d+1}{2}=3^f\frac{3^{e-f}-1}{3^f+1}$, and comparing the 3-adic valuations we get $c=f$, so $3+3^d+3^f+3^{d+f}=2\cdot 3^{e-f}$ which modulo 9 is $3+3^f\equiv 0$, not possible. So $a>b$, and $3^b\frac{3^{a-b}-1}{3^b+1}+3^c\frac{3^d+1}{2}=3^f\frac{3^{e-f}-1}{3^f+1}$. If $b=c$ then the LHS has 3-adic valuation $b=c$ with residue 1, while the RHS has 3-adic valuation $f$ with residue $-1$. So either $b=f<c$ or $c=f<b$. In the first case, we get $2\cdot 3^{a-b}+3^{c-b}(3^b+1)(3^d+1)=2\cdot 3^{e-b}$. In particular, $e>b$ so we must have $a=c$ for the 3-adic valuations to match, and $3+3^b+3^d+3^{b+d}=2\cdot 3^{e-a}$, which modulo 9 is $3+3^b\equiv 2\cdot 3^{e-a}$, so $b=e-a=1$, but then the LHS sum is too big.

If $c=f<b$ then $3^{b-c}\frac{3^{a-b}-1}{3^b+1}+\frac{3^d+1}{2}=\frac{3^{e-c}-1}{3^c+1}$. Modulo 9 this is $-3^{b-c}+5\equiv -1+3^c$, so $3^c+3^{b-c}\equiv 6$ which implies $c=1,b=2$. Then $9+6\cdot 3^{a-2}+10\cdot 3^d=5\cdot 3^{e-1}$, so $e\geq 4$ and $1+2\cdot 3^{a-3}+10\cdot 3^{d-2}=5\cdot 3^{e-3}$, which is not possible modulo 3 since $a$ is even, so $a>3$.

\item $p=5$ and $\frac{5^a+1}{5^b+1}+\frac{5^c+5^{c+d}}{2}=7$, impossible as the LHS is $>7$.
\end{itemize}

{\bf Type (3,3):} Here $A=\frac{p^a+1}{p^b+1}$, $B=p^c\frac{p^d+1}{p^e+1}$, $A+B=\frac{p^a+1}{p^b+1}+\frac{p^c+p^{c+d}}{p^e+1}$. Depending of the type of $A+B$, we have:

\begin{itemize}
\item $\frac{p^a+1}{p^b+1}+\frac{p^c+p^{c+d}}{p^e+1}=p^f(p^g+1)$. Here we always have the pair $\mathbf{(1,1)_p}$ (part of family \#3) if $p>2$, assume we are not in that case. Modulo $p$ we get $1+p^c=p^f+p^{f+g}$. If $p>2$ this implies $f=0$ (so then $g>0$ since the sum of the pair is not 2) and $c>0$. If $p=2$ then $c=0$ or $f=0$ by parity. Suppose $f=0$ first, then $p^b\frac{p^{a-b}-1}{p^b+1}+p^c\frac{p^d+1}{p^e+1}=p^g$. If $a=b$ then $d=e,c=g$ and we get the pairs $\mathbf{(1,p^c)_p}$, part of \#1. Assume $a\geq 3b$, then $b,c<g$ so comparing the $p$-adic valuations we must have $b=c$ and $\frac{p^{a-b}-1}{p^b+1}+\frac{p^d+1}{p^e+1}=p^{g-b}$, so $p^b\frac{p^{a-2b}+1}{p^b+1}+p^e\frac{p^{d-e}-1}{p^e+1}=p^{g-b}$. If $d=e$ then $g=2b,a=3b$ and we get the pairs $\mathbf{(\frac{p^{3b}+1}{p^b+1},p^b)_p}$, part of \#2. Otherwise, arguing as before we get $b=e$ and $\frac{p^{a-2b}+p^{d-b}}{p^b+1}=p^{g-2b}$, so $p^{a-2b}+p^{d-b}=p^{g-2b}+p^{g-b}$.  Then either $d=g=a$ (and we get the pairs $\mathbf{(\frac{p^a+1}{p^b+1},p^b\frac{p^a+1}{p^b+1})_p}$) or $a-b=g,g-b=d$, so $d=a-2b$ (and we get the pairs $\mathbf{(\frac{p^a+1}{p^b+1},p^b\frac{p^{a-2b}+1}{p^b+1})_p}$).

Suppose now $c=0$ (so necessarily $p=2$ and $f>0$ by parity, and we may assume $g>0$). Modulo 4 we get $2\equiv 2^e+2^f+2^b$ (if $a>b$) or $2\equiv 2^e+2^f$ (if $a=b$). In any case, one of $e,f,b$ must be $1$. If $f=1$ then $2^b\frac{2^{a-b}-1}{2^b+1}+2^e\frac{2^{d-e}-1}{2^e+1}=2^{g+1}$. If $a=b$ (or $d=e$) then this is not possible by lemma \ref{quotient}, so assume $a>b,d>e$, then comparing the 2-adic valuations we get $b=e$ and $\frac{2^{a-b}+2^{d-b}-2}{2^b+1}=2^{g+1-b}$, that is, $2^{a-b}+2^{d-b}=2^{g+1}+2^{g+1-b}+2$, which implies $g=b$ and either $a=2b+1$ or $a=b+2$, none of which is possible unless $a=3,b=1$.

If $f>1$ then one of $b,e$ is 1, by symmetry let $b=1$ (so $a\geq 5$). We get $\frac{2^{a-2}+1}{3}+2^{e-2}\frac{2^{d-e}-1}{2^e+1}=2^{f-2}(2^g+1)$. If $d=e$ then $f=2$ and $2\frac{2^{a-3}-1}{3}=2^g$, so $a=5$ by lemma \ref{quotient} and we get the pair $\mathbf{(11,1)_2}$, part of \#14. Otherwise, either $e=2$ or $f=2$ by parity. In the first case, modulo 8 we get $6\equiv 2^{f+g-2}+2^{f-2}$, so $f=3,g=1$, the sum of the pair is 24, and the only option is $\mathbf{(13,11)_2}$. If $f=2$ then $2\frac{2^{a-3}-1}{3}+2^{e-2}\frac{2^{d-e}-1}{2^e+1}=2^g$, and comparing the 2-adic valuations we get $e=3$, $2^2\frac{2^{a-5}-1}{3}+2^3\frac{2^{d-6}+1}{9}=2^{g-1}$, which looking at the 2-adic valuations is impossible unless $a=5,d=9$, which gives the pair $\mathbf{(11,57)_2}$.

\item $p\geq 3$ and $\frac{p^a+1}{p^b+1}+\frac{p^c+p^{c+d}}{p^e+1}=p^f\frac{p^g+1}{2}$. Modulo $p-1$ we get $2\equiv 1$ or $\frac{p+1}{2}$, so $p=3$ and we may assume $g>1$. Modulo $3$ we get $1+3^c\equiv 2\cdot 3^f$, so $c=f=0$, and then $\frac{3^{a-1}-3^{b-1}}{3^b+1}+\frac{3^{d-1}-3^{e-1}}{3^e+1}=\frac{3^{g-1}-1}{2}$. Modulo 3 this is $3^{a-1}-3^{b-1}+3^{d-1}-3^{e-1}\equiv 1$. Since $a=1\Rightarrow b=1$ and $d=1\Rightarrow e=1$, this is only possible if $b=e=1$, and then $3^a+3^d=2\cdot 3^g$, so $a=d$ and we get the pairs $\mathbf{(\frac{3^a+1}{4},\frac{3^a+1}{4})_3}$.

\item $\frac{p^a+1}{p^b+1}+\frac{p^c+p^{c+d}}{p^e+1}=p^f\frac{p^g+1}{p^h+1}$, modulo $p-1$ this is $2\equiv 1$, so it can only happen for $p=2$. By parity, either $c=0$ of $f=0$, suppose $c=0,f>0$ first. If $a=b$ and $d=e$ we get the pair $\mathbf{(1,1)_2}$, part of family \#6. If $d=e$ and $a\geq 3b$, modulo 4 we get $2\equiv 2^b+2^f$, so $b=1$ or $f=1$. In the first case, we get the pairs $\mathbf{(\frac{2^b+1}{3},1)_2}$. In the second case, $2^b\frac{2^{a-b}-1}{2^b+1}=2^{h+1}\frac{2^{g-h}-1}{2^h+1}$, which is impossible by lemma \ref{quotient}. By symmetry, we may assume from now on that $a\geq 3b,d\geq 3e$. Modulo 4 we get $2\equiv 2^b+2^e+2^f$, so one of $b,e,f$ must be 1. Suppose $f=1$ first, then substracting 2 on both sides we get $2^b\frac{2^{a-b}-1}{2^b+1}+2^e\frac{2^{d-e}-1}{2^e+1}=2^{h+1}\frac{2^{g-h}-1}{2^h+1}$, so comparing the 2-adic valuations two of $b,e,h+1$ must coincide. If $b=e$ then $2^{b+1}\frac{2^{a-b-1}+2^{d-e-1}-1}{2^b+1}=2^{h+1}\frac{2^{g-h}-1}{2^h+1}$, so $b=h$ and $2^{a-b-1}+2^{d-b-1}=2^{g-h}$, which implies $a=d$ and we get the pairs $\mathbf{(\frac{2^a+1}{2^b+1},\frac{2^a+1}{2^b+1})_2}$. If $b=h+1<e$ then $2^b\frac{2^{a-2b}+1}{2^b+1}+2^{e-b}\frac{2^{d-e}+1}{2^e+1}=2^h\frac{2^{g-2h}+1}{2^h+1}$, and since $h<b$ we must have $e-b=h$ and $2\frac{2^{a-2b}+1}{2^b+1}+\frac{2^{d-e}-1}{2^e+1}=\frac{2^{g-2h}+1}{2^h+1}$. If $g=3h$ the RHS is 1, not possible since both summands on the LHS are positive. So modulo $2^b$ we get $1\equiv 1-2^h$, which is also not possible since $h<b$. The case $e=h+1<b$ is similar by symmetry.

Still in the $c=0$ case, suppose now that $b=1,f>1$ (the case $e=1$ is similar by symmetry). Then $a\geq 5$, and $2^2\frac{2^{a-2}+1}{3}+2^e\frac{2^{d-e}-1}{2^e+1}=2^f\frac{2^g+1}{2^h+1}$, so $f=2$ or $e=2$. In the first case, substracting 1 we get $\frac{2^{a-3}-1}{3}+2^{e-3}\frac{2^{d-e}-1}{2^e+1}=2^{h-1}\frac{2^{g-h}-1}{2^h+1}$, so either $h=1$ (and then $2^{e-3}\frac{2^{d-e}-1}{2^e+1}=2^{a-3}\frac{2^{g-a+2}-1}{3}$, not possible by lemma \ref{quotient}) or $e=3$ (and then $2^2\frac{2^{a-5}-1}{3}+2^3\frac{2^{d-6}+1}{9}=2^{h-1}\frac{2^{g-h}-1}{2^h+1}$, which is impossible if $a=5$ by lemma \ref{quotient}, so $a>5,h=3$ and $2^{a-4}+2^{a-5}+2^{d-5}=2^{g-3}$, which implies $a=d$ and we get the pairs $\mathbf{(\frac{2^a+1}{3},\frac{2^a+1}{9})_2}$). In the second case, $\frac{2^{a-2}+1}{3}+\frac{2^{d-2}-1}{5}=2^{f-2}\frac{2^g+1}{2^h+1}$, so $\frac{2^{a-3}-1}{3}+2\frac{2^{d-4}+1}{5}=2^{f-3}\frac{2^g+1}{2^h+1}$, and then $f=3$ and modulo 4 we get $3\equiv 1-2^h$, so $h=1$ (and $g\geq 5$) and then modulo 8 we have $2^{a-3}\equiv 4$, so $a=5$, $\frac{2^{d-4}+1}{5}=\frac{2^{g-1}-1}{3}$, impossible by lemma \ref{quotient}.

We move now to the case $c>0,f=0$. Since the sum of the pair is $>1$, we always have $g\geq 3h$. Then $2^b\frac{2^{a-b}-1}{2^b+1}+2^c\frac{2^d+1}{2^e+1}=2^h\frac{2^{g-h}-1}{2^h+1}$. If $a=b$ then lemma \ref{quotient} implies that either $d=e$ (and we get the pairs $(1,2^c)$) or $d=3,e=1,c=h=2$, and we get the pair $\mathbf{(1,12)_2}$, part of \#17. So assume $a\geq 3b$ from now on, then comparing the 2-adic valuations we get that two of $b,c,h$ must coincide (and the third one must be strictly larger). Consider the case $b=c<h$ first. Then $2^b\frac{2^{a-2b}+1}{2^b+1}+2^e\frac{2^{d-e}-1}{2^e+1}=2^{h-b}\frac{2^{g-h}-1}{2^h+1}$. If $d=e$, since $h>1$ by lemma \ref{quotient} we get $b=1,a=5$ which gives the pair $\mathbf{(11,2)_2}$. Otherwise, two of $b,e,h-b$ must coincide. If $b=e<h-b$ then $\frac{2^{a-2b}+2^{d-b}}{2^b+1}=2^{h-2b}\frac{2^{g-h}-1}{2^h+1}$, impossible by lemma \ref{quotient} since $h>2b\geq 2$. If $b=h-b<e$ then $\frac{2^{a-2b}+1}{2^b+1}+2^{e-b}\frac{2^{d-e}-1}{2^e+1}=\frac{2^{g-h}-1}{2^h+1}$. If $a=3b$, modulo 4 this is $1+2^{e-b}\equiv -1$ (since $h=2b\geq 2$), so $e-b=1$, and adding 1 we get $2^{e+1}\frac{2^{d-2e}+1}{2^e+1}=2^h\frac{2^{g-2h}+1}{2^h+1}$, and then by lemma \ref{quotient} we have $b+1=e=h-1=2b-1$, so $b=2,e=3,d=3e=9,a=3b=6$ and we get the pair $\mathbf{(13,228)_2}$. Otherwise, modulo 4 we get $1-2^b-2^{e-b}\equiv -1$, so either $b=1$ or $e-b=1$. The latter would imply $2^b\frac{2^{a-3b}-1}{2^b+1}+2^{b+2}\frac{2^{d-2e}+1}{2^e+1}=2^h\frac{2^{g-2h}+1}{2^h+1}$, which is not possible since $b<b+2,h$. So $b=1,h=2$ and $\frac{2^{a-2}+1}{3}+2^{e-1}\frac{2^{d-e}-1}{2^e+1}=\frac{2^{g-2}-1}{5}$, which adding 1 gives $\frac{2^{a-4}+1}{3}+2^{e-3}\frac{2^{d-e}-1}{2^e+1}=\frac{2^{g-4}+1}{5}$, and modulo 4 this is $2^{a-4}+2\equiv 2^{e-3}$, so either $a=5$ (which would imply $2^{e-5}\frac{2^{d-e}-1}{2^e+1}=2^2\frac{2^{g-6}-1}{5}$, not possible by lemma \ref{quotient}) or $e=4$ (and then modulo 8 we get $3\cdot 2^{a-4}\equiv 5\cdot 2^{g-4}+4$, and since $a\geq 5$ is odd and $g\geq 6$, we must have $g=6$ and the sum of the pair is 13, which leaves $(11,2)$ as the only option, already found above). Finally, if $e=h-b<b$ (so $b\geq 2,h\geq 3$) we get $2^{b-e}\frac{2^{a-2b}+1}{2^b+1}+2^e\frac{2^{d-2e}+1}{2^e+1}=2^h\frac{2^{g-2h}+1}{2^h+1}$, and since $e<h$ this implies $b=2e$ and $\frac{2^{a-2b}+1}{2^b+1}+\frac{2^{d-2e}+1}{2^e+1}=2^b\frac{2^{g-2h}+1}{2^h+1}$, which modulo 4 is $2-2^e\equiv 0$, so $e=1,b=2,h=3$ and $\frac{2^{a-6}-1}{5}+\frac{2^{d-4}+1}{3}=\frac{2^{g-6}+1}{9}$, and then by parity $a=6$ and by lemma \ref{quotient} $d=5$, which gives the pair $\mathbf{(13,44)_2}$.

Consider now the case $b=h<c$. Then $2^c\frac{2^d+1}{2^e+1}=2^a\frac{2^{g-h+b-a}-1}{2^b+1}$, so by lemma \ref{quotient} we get $d=e,a=c$, which gives the pairs $\mathbf{(\frac{2^a+1}{3},2^a)_2}$.

Finally, consider the case where $c=h<b$. Then $2^{b-c}\frac{2^{a-b}-1}{2^b+1}+\frac{2^d+1}{2^e+1}=\frac{2^{g-h}-1}{2^h+1}$. Modulo 4 this is $2\equiv 2^{b-c}+2^e+2^h$ (or $2\equiv 2^{b-c}+2^h$ if $d=e$) so one of $b-c,e,h$ must be $1$. If $b-c=1$ then $2^{b+1}\frac{2^{a-2b}+1}{2^b+1}+2^e\frac{2^{d-e}-1}{2^e+1}=2^h\frac{2^{g-2h}-1}{2^h+1}$. Since $h<b$, this is impossible if $d=e$, and otherwise it implies that $e=h$, in which case $2^{b+1-e}\frac{2^{a-2b}+1}{2^b+1}+\frac{2^{d-e}-1}{2^e+1}=\frac{2^{g-2h}+1}{2^h+1}$, which modulo 4 gives $2^{e+1}\equiv 2$, not possible. If $e=1,b-c>1$ we may assume $d\geq 5$, then $2^{b-c}\frac{2^{a-b}-1}{2^b+1}+\frac{2^d+1}{3}=\frac{2^{g-h}-1}{2^h+1}$. Modulo 4 we get $h\geq 2$, and then $2^{b-c-2}\frac{2^{a-b}-1}{2^b+1}+\frac{2^{d-2}+1}{3}=2^{h-2}\frac{2^{g-2h}+1}{2^h+1}$, so $b-c=2$ or $h=2$. In the first case, $2^b\frac{2^{a-2b}+1}{2^b+1}+2\frac{2^{d-3}-1}{3}=2^{h-2}\frac{2^{g-2h}+1}{2^h+1}$, so $h=c=3$ (since $b\geq 2$), and then $b=c+2=5$ and $2^4\frac{2^{a-10}+1}{2^5+1}+2^2\frac{2^{d-5}-1}{3}=2^3\frac{2^{g-9}-1}{9}$, which is impossible looking at the 2-adic valuations of the terms. In the second case, $2^{b-c-2}\frac{2^{a-b}-1}{2^b+1}+2\frac{2^{d-3}-1}{3}=2^2\frac{2^{g-6}-1}{5}$, so $b=c+3=5$, $\frac{2^{a-5}-1}{33}+\frac{2^{d-3}-1}{3}=2\frac{2^{g-6}-1}{5}$ and modulo 4 we get $0\equiv 2$, impossible. Finally, if $h=1$ and $e,b-c>1$ we may assume $g\geq 5$, and we have $2^{b-1}\frac{2^{a-b}-1}{2^b+1}+2^e\frac{2^{d-e}-1}{2^e+1}=2^2\frac{2^{g-3}-1}{3}$, so either $e=2$ or $b=3$. In the first case we get $2^{b-3}\frac{2^{a-b}-1}{2^b+1}+2^2\frac{2^{d-4}+1}{5}=2\frac{2^{g-4}+1}{3}$, so $b=4$ and $2^4\frac{2^{a-8}+1}{17}+2\frac{2^{d-4}+1}{5}=2^2\frac{2^{g-6}+1}{3}$, not possible looking at their 2-adic valuations. In the second case $\frac{2^{a-b}-1}{2^b+1}+2^{e-2}\frac{2^{d-e}-1}{2^e+1}=\frac{2^{g-3}-1}{3}$, which modulo 4 implies $e=3$, and then adding 3 on both sides we get $2^b\frac{2^{a-2b}+1}{2^b+1}+2^4\frac{2^{d-6}+1}{9}=2^3\frac{2^{g-6}+1}{3}$, so $b=3$ and $2^{a-6}+2^{d-5}=2^{g-6}+2^{g-5}$, and either $a=d$ (which gives the pairs $\mathbf{(\frac{2^a+1}{9},2\frac{2^a+1}{9})_2}$) or $a=d+2$ (not possible since $a$ and $d$ must be multiples of 3).

\item $p=5$ and $\frac{5^a+1}{5^b+1}+\frac{5^c+5^{c+d}}{5^e+1}=7\cdot 5^f$. Modulo 5 this is $1+5^c=2\cdot 5^f$, so $c=f=0$ and $\frac{5^a+1}{5^b+1}+\frac{5^d+1}{5^e+1}=7$. Since both summands are 1 modulo 5, the only option is the already found pair $(6,1)_5$, of type (1,3).
\end{itemize}

{\bf Type (3,4):} Here $p=5$, $A=\frac{5^a+1}{5^b+1}$, $B=7\cdot 5^c$, $A+B=\frac{1+5^a+7\cdot 5^c+7\cdot 5^{b+c}}{1+5^b}$. Modulo 4 we have $A\equiv 1$ and $B\equiv 3$, so $A+B\equiv 0$ which is not possible for any FM exponent. 

{\bf Type (4,1):} Here $p=5$, $A=7$, $B=5^a(5^b+1)$, $A+B=7+5^a+5^{a+b}$. We may assume $a>0$, since otherwise we get a reversed pair of type (1,4), and then $A+B$ is prime to 5. Modulo 4 we have $A\equiv 3$, $B\equiv 2$, so $A+B\equiv 1$ can only be of types 2 and 3. Depending of the type of $A+B$, we have:

\begin{itemize}
\item $7+5^a+5^{a+b}=\frac{5^c+1}{2}$, which modulo 5 is $2\equiv 3$, not possible.

\item $7+5^a+5^{a+b}=\frac{5^c+1}{5^d+1}$, which modulo 5 is $2\equiv 1$, not possible.
\end{itemize}

{\bf Type (4,2):} Here $p=5$, $A=7$, $B=5^a\frac{5^b+1}{2}$, $A+B=\frac{14+5^a+5^{a+b}}{2}$. We may assume $a>0$, since otherwise we get a reversed pair of type (2,4), and then $A+B$ is prime to 5. Modulo 4 we have $A\equiv 3$, $B\equiv 1$ or $3$, so $A+B\equiv 0$ or $2$ can only be of type 1. Then $\frac{14+5^a+5^{a+b}}{2}=1+5^c$, which modulo 5 is $2\equiv 1+5^c$, only possible if $c=0$, not possible since $2<7$.

{\bf Type (4,3):} Here $p=5$, $A=7$, $B=5^a\frac{5^b+1}{5^c+1}$. Modulo 4 we have $A\equiv 3$, $B\equiv 1$, so $A+B\equiv 0$ can never be an FM-exponent.

{\bf Type (4,4):} Here $p=5$, $A=7$, $B=7\cdot 5^a$. Modulo 4 we have $A\equiv B\equiv 3$, so $A+B\equiv 2$ can only be of type 1: $7+7\cdot 5^a=5^b+1$. Modulo 5 this is $2+2\cdot 5^a=5^b+1$, only possible for $b=0$, but then $5^b+1=2<7$.

\end{proof}

For each of the pairs in the previous theorem, we will now check whether or not the corresponding local system has finite monodromy. We start with two technical lemmas.

 \begin{lemma}\label{vfunction}
 Fix a prime $p$, and let $a\geq 3b$ be an odd multiple of $b$ and $c\geq b$. Then
 \begin{enumerate}
     \item $V_p\left(\frac{p^a+1}{(p^b+1)(p^c-1)}\right)=\frac{a}{2c}-\frac{b}{2c}+\frac{1}{c(p-1)}$
     \item $V_p\left(\frac{p^{a+b}-p^b-2}{(p^b+1)(p^c-1)}\right)=\frac{a}{2c}+\frac{b}{2c}-\frac{1}{c(p-1)}$
 \end{enumerate}
 \end{lemma}
 \begin{proof}\begin{enumerate}
    \item Since $\frac{p^a+1}{p^b+1}<p^{a-b}<p^a-1\leq p^c-1$, $V_p\left(\frac{p^a+1}{(p^b+1)(p^c-1)}\right)$ is $\frac{1}{c(p-1)}$ times the sum of the $p$-adic digits of $\frac{p^a+1}{p^b+1}=\sum_{k=1}^{(a-b)/2b}p^{(2k-1)b}(p^b-1)+1$, which is $\frac{a-b}{2b}b(p-1)+1$.
    \item It follows from (1) since $\frac{p^a+1}{p^b+1}+\frac{p^{a+b}-p^b-2}{p^b+1}=p^a-1$, all whose last $a$ $p$-adic digits are $p-1$.
     \end{enumerate}
 \end{proof}

 \begin{lemma}\label{switchsum}
 Let $k=\mathbb F_{2^r}$ and $\psi:k\to\mathbb C^\times$ the additive character $t\mapsto (-1)^{\mathrm{Tr}_{k/\mathbb F_2}(t)}$. Then for every $t,y\in k$ we have the equality
 \[
 \sum_{\substack{x\in k\\ x^2+x=y}}\psi(tx)=\sum_{\substack{u\in k\\ u^2+u=t^2}}\psi(uy).
 \]
 \end{lemma}

 \begin{proof}

 Recall that an equation of the form $w^2+w=z$ for a given $z\in k$ has a root in $k$ if and only if $\psi(z)=1$, and in this case, if $w_0$ is such a root, then $w_0+1$ is the other one.

 It follows that $(1-\psi(t))\sum_{x^2+x=y}\psi(tx)=0$ and $(1-\psi(y))\sum_{u^2+u=t^2}\psi(uy)=0$. Therefore, if $\psi(t)=\psi(t^2)=-1$ or $\psi(y)=-1$ both sums vanish. Assume now that $\psi(t)=\psi(y)=1$ and let $x_0\in k$ (resp.~$u_0\in k$) be a root of the equation $x^2+x=y$ (resp.~$u^2+u=t^2$). Then $\sum_{x^2+x=y}\psi(tx)=(1+\psi(t))\psi(tx_0)=2\psi(tx_0)$ and $\sum_{u^2+u=t^2}\psi(uy)=(1+\psi(y))\psi(u_0y)=2\psi(u_0y)$. The equality $\psi(tx_0)=\psi(t^2x_0^2)=\psi(u_0^2x_0^2)\psi(u_0x_0^2)=\psi(u_0x_0)^2\psi(u_0y)=\psi(u_0y)$ proves the lemma.
 \end{proof}
 
\begin{thm} Let $A,B$ be positive integers such that $A$ is prime to $p$ and the polynomial $f(x)=x^A(x-1)^B$ is such that $\mathcal{F}_f$ has finite monodromy group. Then $(A,B)$ is one of the following pairs (or their reversed pairs):
\begin{enumerate}
\item $p\geq 2$, $(1,p^a)$ for $a\geq0$.
\item $p=2$, $(1,12)$.
\item $p=2$, $(2^a+1,2^a+1)$ for $a\geq 1$.
\item $p=2$, $(\frac{2^a+1}{2^b+1},\frac{2^a+1}{2^b+1})$ for $b\geq 1$ and $a\geq b$ an odd multiple of $b$.
\item $p=2$, $(1,2^a+1)$ for $a\geq 1$.
\item $p=2$, $(2^a+1,2^a)$ for $a\geq 1$.
\item $p=2$, $(2^a+1,\frac{2^a+1}{3})$ for odd $a\geq1$.
\item $p=2$, $(1,\frac{2^a+1}{3})$ for odd $a\geq1$.
\item $p=2$, $(\frac{2^a+1}{3},2^a)$ for odd $a\geq 1$.
\item $p=3$, $(1,4)$, $(1,6)$, $(2,2)$ or $(4,3)$.
\item $p=3$, $(3^a+1,\frac{3^a+1}{2})$ for $a\geq 0$.
\item $p=3$, $(1,\frac{3^a+1}{2})$ for $a\geq 0$.
\item $p=3$, $(\frac{3^a+1}{2},3^a)$ for $a\geq 0$.
\item $p=5$, $(2,1)$.
\end{enumerate}
\end{thm}

\begin{proof} For each item from Theorem \ref{FMpairs} we either prove the corresponding family has finite monodromy or give a counterexample for the $V$-criterion Theorem \ref{Belyicrit}. Write $W_p(d,e,x,y):=V_p(x)+V_p(y)+V_p(y-(d+e)x)+V_p(ex-y)+V_p(-ex)$ for $d,e\in\mathbb Z$.

\begin{enumerate}
\item $p\geq2$, $(\frac{p^a+1}{p^b+1},p^b\frac{p^a+1}{p^b+1})$ for $b\geq 1$ and $a\geq b$ an odd multiple of $b$. Observe that $\frac{p^a+1}{p^b+1}+p^b\frac{p^a+1}{p^b+1}=p^a+1$.

    If $a=b$ we have the pair $(1,p^a)$ with finite monodromy. Indeed, the trace function of $\mathcal{F}_f$ with $f(x)=x^{p^a+1}-x$ is: \[(k_r;s,t) \mapsto -\frac{1}{q^{r/2}}\sum_{x\in k_r}\psi_r(sx^{p^a+1}+(t-s)x),\] 
    which is an algebraic integer since $p^a+1$ is an FM-exponent. Observe that this argument works for $a=0$ as well.

    Assume $a\geq 3b$ and consider $(x,y)=\left(\frac{1}{p^a-1},\frac{2}{p^a-1}\right)$. Then
    $$
    V_p(x)+V_p(y)+V_p(y-(p^a+1)x)+V_p\left(p^b\frac{p^a+1}{p^b+1}x-y\right)+V_p\left(-p^b\frac{p^a+1}{p^b+1}x\right)=
    $$
    $$
    =\frac{1}{a(p-1)}+V_p\left(\frac{2}{p^a-1}\right)+1+\frac{b}{a}-\frac{2}{a(p-1)}
    $$
    by lemma \ref{vfunction}. If $p=2$ this is $1+\frac{b}{a}\leq \frac{4}{3}<\frac{3}{2}$. If $p>2$ it is $1+\frac{b}{a}+\frac{1}{a(p-1)}\leq \frac{4}{3}+\frac 1{a(p-1)}<\frac{3}{2}$ unless $a(p-1)\leq 6$, which can only happen for $a=p=3$, in which case we have $W_3(7,21,\frac{1}{8},\frac{1}{2})=\frac{5}{4}<\frac{3}{2}$.
    

\item $p\geq2$, $(\frac{p^{a+2b}+1}{p^b+1},p^b\frac{p^a+1}{p^b+1})$. Observe that $d+e=p^{a+b}+1$.
    
    If $a=b$ then $d=\frac{p^{3b}+1}{p^b+1}$, $e=p^b$ and $d+e=p^{2b}+1$. Consider $(x,y)=(\frac{p^b-1}{p^{3b}-1},p^b\frac{1}{p^{3b}-1})$. Then \[\begin{aligned}V_p(x)&+V_p(y)+V_p(y-(p^{2b}+1)x)+V_p(p^bx-y)+V_p(-p^bx)=\\ &=1+\frac{1}{3b(p-1)}+V_p\left(\frac{p^{2b}}{p^{3b}-1}\right)+V_p\left(\frac{p^b-2}{p^{3b}-1}\right)=\frac{4}{3}+\frac{1}{3b(p-1)}.\end{aligned}\] This is $<3/2$ unless $b(p-1)\leq 2$, which can only happen for $(p,b)=(2,1),(2,2)$ or $(3,1)$. For $p=2$ and $a=b=1$ we obtain the pair $(3,2)_2$ with finite monodromy. Indeed, $f(x)=x^5+x^3$ and the trace function of $\FF_f$ is \[(k_r; s,t) \mapsto -\frac{1}{q^{r/2}}\sum_{x\in k_r}\psi_r(sx^5+sx^3+tx),\]
    which is a specialization of the finite monodromy family $\FF(5,3,1;\mathbf 1)$ in \cite[Theorem 11.2.3(ii)]{KT24}. For the other two pairs, $W_2(13,4,\frac{19}{255},\frac{4}{15})=W_3(7,3,\frac{11}{80},\frac{3}{8})=\frac{11}{8}<\frac{3}{2}$. 
    
    Now assume $a\geq 3b$ and consider $(x,y)=(\frac{1}{p^{a+b}-1},\frac{2}{p^{a+b}-1})$. Then \[\begin{aligned}V_p(x)&+V_p(y)+V_p(y-(p^{a+b}+1)x)+V_p(ex-y)+V_p(-ex)=\\ &= 1+\frac{b}{a+b}-\frac{1}{(a+b)(p-1)}+V_p\left(\frac{2}{p^{a+b}-1}\right)\end{aligned}\] by lemma \ref{vfunction}. Since $a\geq 3b,$ for $p=2$ this is $=1+\frac{b}{a+b}\leq \frac{5}{4}<\frac{3}{2}$ and for $p>2$ this is $=1+\frac{b}{a+b}+\frac{1}{(a+b)(p-1)}\leq\frac{5}{4}+\frac{1}{4b(p-1)}<\frac{3}{2}$.

\item $p\geq3$, $(\frac{p^a+1}{2},\frac{p^a+1}{2})$ for $a\geq0.$ Observe that $d+e=p^a+1$.

    If $a=0$ we already know these pairs have finite monodromy.

    Assume from now on $a\geq 1$. Consider $(x,y)=(\frac{1}{p^a-1},\frac{2}{p^a-1})$. Then, \[\begin{aligned}V_p(x)&+V_p(y)+V_p(y-(p^a+1)x)+V_p(ex-y)+V_p(-ex)=\\ &=\frac{1}{a(p-1)}+\frac{2}{a(p-1)}+0+\frac{1}{2}-\frac{1}{a(p-1)}+\frac{1}{2}-\frac{1}{a(p-1)}=1+\frac{1}{a(p-1)}.\end{aligned}\] This is $<\frac{3}{2}$ unless $a(p-1)\leq 2$, which is only possible for $p=3,a=1$. For $p=3$ and $a=1$ we find the pair $(2,2)_3$ and it gives finite monodromy. The associated polynomial is $f(x)=(x^2-x)^2=x^4+x^3+x^2$ and the trace function of $\mathcal{F}_f$ takes the form:
    
    \[(k_r;s,t) \mapsto -\frac{1}{q^{r/2}}\sum_{x\in k_r}\psi_r(sx^4+sx^3+sx^2+tx) = \]
    \[=-\frac{1}{q^{r/2}}\sum_{x\in k_r}\psi_r(sx^4+sx^2+(t+s^{1/3})x),\] 
    which is a specialization of the trace function of local system  $\FF(4,2,1;\mathbf 1)$ from \cite[Theorem 11.2.3(i-bis)]{KT24}.

\item Observe that $(5,17)\equiv(5,52)\equiv(33,10)\equiv(33,24)\equiv(5,3)\pmod{2^3-1}$ and $(17,26)\equiv(17,40)\equiv(3,5)\pmod{2^3-1}$. Taking $x=y=\frac{1}{7}$ we get $W(5,3,\frac{1}{7},\frac{1}{7})=\frac{4}{3}<\frac{3}{2}$.

For the remaining pairs, we get $W_2(5,6,\frac{1}{7},\frac{4}{7})=W_2(171,34,\frac{1}{7},\frac{2}{7})=\frac{4}{3}$; \newline $W_2(3,10,\frac{3}{31},\frac{8}{31})=W_2(5,8,\frac{3}{31},\frac{8}{31})=W_2(9,4,\frac{3}{31},\frac{8}{31})=W_2(9,2,\frac{3}{31},\frac{2}{31})=W_2(33,172,\frac{5}{31},\frac{2}{31})=\frac{7}{5}$; \newline
$W_2(9,11,\frac{5}{63},\frac{37}{63})=W_2(9,13,\frac{3}{63},\frac{3}{63})=W_2(9,34,\frac{3}{63},\frac{3}{63})=W_2(11,2,\frac{5}{63},\frac{2}{63})=\frac{4}{3}$; \newline
$W_2(9,17,\frac{3}{31},\frac{16}{31})=W_2(9,48,\frac{3}{31},\frac{16}{31})=W_2(9,43,\frac{3}{31},\frac{1}{31})=W_2(205,36,\frac{1}{31},\frac{1}{31})=\frac{7}{5}$; \newline
$W_2(11,13,\frac{1}{15},\frac{9}{15})=W_2(11,57,\frac{1}{15},\frac{8}{15})=W_2(13,44,\frac{1}{15},\frac{12}{15})=W_2(13,228,\frac{1}{15},\frac{1}{15})=W_2(33,208,\frac{1}{15},\frac{1}{15})=W_2(65,176,\frac{1}{15},\frac{1}{15})=\frac{5}{4}$.

\item[(5-6)] $p=2$, $(a,a)$ for $a$ an arbitrary FM-exponent. These pairs have finite monodromy: here $f(x)=x^a(x+1)^a$, and the trace function of $\FF_f$ is
\[
(k_r;s,t)\mapsto -\frac{1}{q^{r/2}}\sum_{x\in k_r}\psi_r(sx^a(x+1)^a+tx)=
\]
\[
=-\frac{1}{q^{r/2}}\sum_{y\in k_r}\psi_r(sy^a)\sum_{x^2+x=y}\psi_r(tx)=-\frac{1}{q^{r/2}}\sum_{y\in k_r}\psi_r(sy^a)\sum_{u^2+u=t^2}\psi_r(uy)=
\]
\[
-\frac{1}{q^{r/2}}\sum_{u^2+u=t^2}\sum_{y\in k_r}\psi_r(sy^a+uy)
\]
by lemma \ref{switchsum}. Since $a$ is an FM-exponent, each $-\frac{1}{q^{r/2}}\sum_{y\in k_r}\psi_r(sy^a+uy)$ term is an algebraic integer.

\setcounter{enumi}{6}

\item $p=2$, $(1,2^a+1)$ for $a\geq1$. These pairs have finite monodromy. Indeed, we work with the reversed pair $(d,e)=(2^a+1,1)$ so $f(x)=x^{2(2^{a-1}+1)}+x^{2^a+1}$ and the trace function of $\mathcal{F}_f$ is 
    
    \[(k_r;s,t) \mapsto -\frac{1}{q^{r/2}}\sum_{x\in k_r}\psi_r(sx^{2(2^{a-1}+1)}+sx^{2^a+1}+tx)=\] \[=-\frac{1}{q^{r/2}}\sum_{x\in k_r}\psi_r(sx^{2^a+1}+s^{1/2}x^{2^{a-1}+1}+tx)\]
    which is a specialization of the trace function of the sheaf $\FF(2^{a+1}+1,2^a+1,1;\mathbf 1)$ in \cite[Theorem 11.2.3(ii)]{KT24}.
    
\item $p=2$, $(2^a+1,2^a)$ for $a\geq1$. These pairs have finite monodromy. Indeed, $f(x)=x^{2^{a+1}+1}+x^{2^a+1}$ and the trace function of $\mathcal{F}_f$ is \[(k_r;s,t)\mapsto -\frac{1}{q^{r/2}}\sum_{x\in k_r} \psi_r(sx^{2^{a+1}+1}+sx^{2^a+1}+tx),\]
    which is a specialization of the finite monodromy family $\FF(2^{a+1}+1,2^a+1,1;\mathbf 1)$ in \cite[Theorem 11.2.3(ii)]{KT24}.
    
\item $p=2$, $(3,2^a+1)$ for $a\geq 1$. Assume $a\geq 1$ is odd. If $a=1$
    the pair has finite monodromy by \#5. If $a=3$ we also obtain a pair $(3,9)_2$ with finite monodromy since $f(x)=x^3(x+1)(x+1)^8=x^{12}+x^{11}+x^4+x^3$ and the trace function of $\mathcal{F}_f$ is \[(k_r;s,t) \mapsto -\frac{1}{q^{r/2}}\sum_{x\in k_r}\psi_r(sx^{11}+(s+s^{1/4})x^3+(t+s^{1/4})x),\] which is a specialization of the finite monodromy family $\FF(11,3,1;\mathbf 1)$ in \cite[Theorem 11.2.3(iii)]{KT24}.

    For $a\geq 5$, we show that the pair $(2^a+1,3)$ does not have finite monodromy. To this end, consider $x=y=\frac{2^{3\lfloor\frac{a-2}{3}\rfloor}-1}{(2^3-1)(2^{a-2}-1)}=\frac{\sum_{k=0}^{\lfloor\frac{a-2}{3}\rfloor-1}2^{3k}}{2^{a-2}-1}.$ Observe that $3x=x+2x$ and there are no carries while adding the numerators of $x$ and $2x$, so $V_2(3x)=V_2(x)+V_2(2x)$. Then 
    
    \[\begin{aligned}V_2(x)&+V_2(x)+V_2((1-(2^a+4))x)+V_2(2x)+V_2(-3x)=\\ & V_2\left(-\frac{2^{3\left\lfloor\frac{a-2}{3}\right\rfloor}-1}{2^3-1}\frac{2^a+2+1}{2^{a-2}-1}\right)+1+\frac{\left\lfloor\frac{a-2}{3}\right\rfloor}{a-2}.\end{aligned}\] 
    
    Observe that $V_2\left(-\frac{2^{3\left\lfloor\frac{a-2}{3}\right\rfloor}-1}{2^3-1}\frac{2^a+2+1}{2^{a-2}-1}\right)=V_2\left(-\frac{2^{3\left\lfloor\frac{a-2}{3}\right\rfloor}-1}{2^{a-2}-1}\right)=1-\frac{3\left\lfloor\frac{a-2}{3}\right\rfloor}{a-2}$, 
    so $W_2(d,e,x,y)=2-\frac{2\left\lfloor\frac{a-2}{3}\right\rfloor}{a-2}$ and this is $<3/2$ unless $a=7$. For $a=7$, we get $W_2(129,3,\frac{5}{31},\frac{9}{31})=\frac{7}{5}<\frac{3}{2}$.

    Now assume $a\geq2$ is even. We show the pair $(2^a+1,3)$ does not give finite monodromy. Let $(x,y)= (\frac{2^a-1}{3(2^{a+1}-1)},\frac{2^{a-1}-1}{2^{a+1}-1})=(\frac{\sum_{k=0}^{a/2-1}2^{2k}}{2^{a+1}-1},\frac{\sum_{k=0}^{a-2}2^k}{2^{a+1}-1})$. Then \[\begin{aligned}&V_2(x)+V_2(y)+V_2(y-(d+e)x)+V_2(ex-y)+V_2(-ex)=\\ &= \frac{3a+2}{2(a+1)}+V_2\bigg(-(2^a+4)\frac{\sum_{k=0}^{a/2-1}2^{2k}}{2^{a+1}-1}+\frac{\sum_{k=0}^{a-2}2^k}{2^{a+1}-1}\bigg)=\frac{3a+2}{2(a+1)}< \frac{3}{2},\end{aligned}\] because $(2^a+4)\sum_{k=0}^{a/2-1}2^{2k}-\sum_{k=0}^{a-2}2^k=\sum_{k=0}^{a/2-1}2^{2(k+a/2)}+\sum_{k=1}^{a/2}2^{2k}-\sum_{k=0}^{a/2-1}2^{2k}-\sum_{k=1}^{a/2-1}2^{2k-1}=\sum_{k=1}^{a/2-1}2^{2(k+a/2)}-\sum_{k=1}^{a/2-1}2^{2k-1}+2^{a+1}-1\equiv 0\pmod{2^{a+1}-1}.$

\item $p=2$, $(2^a+1,3\cdot2^a)$ for $a\geq1$. Assume $a\geq 2$ is even. For $a=2$, $W_2(5,12,\frac{19}{127},\frac{69}{127})=\frac{10}{7}<\frac{3}{2}$. For $a\geq 4$ consider $(x,y)=(\frac{2^{a-2}-1}{3(2^{a-1}-1)},2\frac{2^{a-3}-1}{2^{a-1}-1})=(\frac{\sum_{k=0}^{a/2-2}2^{2k}}{2^{a-1}-1},2\frac{\sum_{k=0}^{a-4}2^k}{2^{a-1}-1})$. Observe that $(2^{a-1}-1)(y-(d+e)x)=2(2^{a-3}-1)-(2^{a+2}+1)\sum_{k=0}^{a/2-2}2^{2k}\equiv2^{a-2}-2-(2^3+1)\sum_{k=0}^{a/2-2}2^{2k}=2^{a-2}-2^{a-1}-\sum_{k=0}^{a-3}2^k=-2^{a-1}+1\equiv 0 \pmod{2^{a-1}-1}.$ Then
\[\begin{aligned}V_2(&x)+V_2(y)+V_2(y-(2^{a+2}+1)x)+V_2(3\cdot2^a\cdot x-y)+V_2(-3x)=\\ &=\frac{a-2}{2(a-1)}+\frac{a-3}{a-1}+0+\frac{1}{a-1}+\frac{1}{a-1}=\frac{3a-4}{2(a-1)}=\frac{3}{2}-\frac{1}{2(a-1)}<\frac{3}{2}.\end{aligned}\]

Since $(2^a+1,3\cdot2^a)\equiv(3,6)\pmod{2^{a-1}-1}=(2^a+1,3\cdot 2^a)\rvert_{a=1}$ we see that the previous argument also shows that for $a=1$ the pair does not give finite monodromy.

    Assume now $a\geq 3$ is odd. Take $x=\frac{2^{a+1}-1}{3(2^{a+2}-1)}=\frac{\sum_{k=0}^{\frac{a-1}{2}}2^{2k}}{2^{a+2}-1},y=2x$. Since $3\cdot2^a\cdot\frac{2^{a+1}-1}{3}-2\frac{2^{a+1}-1}{3}\equiv-\left(2^{a-1}+\sum_{k=0}^{(a-1)/2}2^{2k+1}\right)\pmod{2^{a+2}-1},$ we get
    \[\begin{aligned}2V_2(x)&+V_2(x(1-2^{a+2}))+V_2(ex-2x)+V_2(-ex)=\\ &=\frac{a+1}{a+2}+0+\frac{a+1}{2(a+2)}+\frac{1}{a+2}=\frac{3a+5}{2(a+2)}<\frac{3}{2}.\end{aligned}\]

\item $p=2$, $(2^a+1,\frac{2^{3a}+1}{2^a+1})$ for $a\geq 1.$ Consider the reversed pair $(d,e)=(\frac{2^{3a}+1}{2^a+1},2^a+1)$ and observe that $d+e=2(2^{2a-1}+1)$.

    For $a=1$ the pair has finite monodromy by \#5.

    For $a\geq 2$ take $(x,y)=(\frac{3}{2^{2a}-1},\frac{9}{2^{2a}-1})$. Since $9-(d+e)3=9-(2^{2a}+2)3\equiv0\pmod{2^{2a}-1}$ and $ex-y=3(2^a+1)-9=2\cdot3(2^{a-1}-1)=2(1+\sum_{k=2}^{a-2}2^k+2^a)$ it follows that 
    \[\begin{aligned}V_2(x)&+V_2(y)+V_2(y-(2^{2a}+2)x)+V_2(ex-y)+V_2(-ex)\\ &= \frac{1}{a}+\frac{1}{a}+0+\frac{a-1}{2a}+\frac{a-2}{a}= 1+\frac{a-1}{2a}=\frac{3a-1}{2a}<\frac{3}{2}.\end{aligned}\]

\item $p=2$, $(\frac{2^{3a}+1}{2^a+1},2^a(2^a+1))$ for $a\geq 1$. Observe that $e=2^{2a}+2^a\equiv2^{a}+1\pmod{2^{2a}-1}$ so we conclude using the previous item for $a\geq 2$.
    
    For $a=1$, the pair $(3,6)$ does not have finite monodromy by \#10.

\item $p=2$, $(2^a+1,\frac{2^a+1}{3})$ for odd $a\geq 1$. For $a=1$ the pair has finite monodromy by \#7.
    
    Now we show that the reversed pair $(d,e)=(\frac{2^a+1}{3},2^a+1)$ has finite monodromy for $a\geq 3$. Indeed, $f(x)=x^d(x+1)^e=x^{2^2\frac{2^a+1}{3}}+x^{\frac{2^{a+2}+1}{3}}+x^{2^2\frac{2^{a-2}+1}{3}}+x^{\frac{2^a+1}{3}}$ and the trace function of $\mathcal{F}_f$ is as follows: \[(k_r;s,t) \mapsto -\frac{1}{q^{r/2}}\sum_{x\in k_r} \psi_r(sx^\frac{2^{a+2}+1}{3}+(s+s^{1/4})x^\frac{2^a+1}{3}+s^{1/4}x^\frac{2^{a-2}+1}{3}+tx),\] which is a specialization of the finite monodromy family $\FF(\frac{2^{a+2}+1}{3},\frac{2^a+1}{3},\frac{2^{a-2}+1}{3},1;\mathbf 1)$ by \cite[Theorem 11.2.3(iii)]{KT24}.

\item $p=2$, $(1,\frac{2^a+1}{3})$ for odd $a\geq1$. For $a=1,3$ we have already seen that the pair has finite monodromy. The remaining values $a\geq 5$ have finite monodromy as well since the 
    polynomial of the reversed pair $(d,e)=(\frac{2^a+1}{3},1)$ is $f(x)=x^{2^2\frac{2^{a-2}+1}{3}}+x^\frac{2^a+1}{3}$ and the trace function of $\mathcal{F}_f$ is given by 
    
    \[(k_r;s,t)\mapsto -\frac{1}{q^{r/2}}\sum_{x\in k_r}\psi_r(sx^\frac{2^a+1}{3}+s^{1/4}x^\frac{2^{a-2}+1}{3}+tx),\] which is a specialization of the finite monodromy family $\FF(\frac{2^a+1}{3},\frac{2^{a-2}+1}{3},1;\mathbf 1)$ in \cite[Theorem 11.2.3(iii)]{KT24}.

\item $p=2$, $(\frac{2^a+1}{3},2^a)$ for $a\geq 1$ odd. These pairs have finite monodromy. Indeed, 
    $f(x)=x^\frac{2^{a+2}+1}{3}+x^\frac{2^a+1}{3}$ and the trace function of $\mathcal{F}_f$ is \[(k_r;s,t)\mapsto -\frac{1}{q^{r/2}}\sum_{x\in k_r}\psi_r(sx^\frac{2^{a+2}+1}{3}+sx^\frac{2^a+1}{3}+tx),\]  which is a specialization of the finite monodromy family $\FF(\frac{2^{a+2}+1}{3},\frac{2^{a}+1}{3},1;\mathbf 1)$ in \cite[Theorem 11.2.3(iii)]{KT24}.
    
\item $p=2$, $(3,\frac{2^{2a}+1}{5})$ for $a\geq1$ odd. For $a=1$ the pair has finite monodromy by \#7.
    
    For $a\geq3$, consider $x=y=\frac{1}{2^{2a-1}-1}$. Observe that $(d+e)-1=e+2=1+2+(2^2-1)\sum_{k=1}^{\frac{a-1}{2}}2^{2(2k-1)}=1+2+(2+1)\sum_{k=1}^{\frac{a-1}{2}}2^{2(2k-1)}$. It follows that 
    \[\begin{aligned}2V_2(&x)+V_2((1-(d+e))x)+V_2((e-1)x)+V_2(-ex)\\ &=\frac{2}{2a-1}+\frac{a-2}{2a-1}+\frac{a-1}{2a-1}+\frac{a-1}{2a-1}=\frac{3a-2}{2a-1}<\frac{3}{2}.\end{aligned}\]

\item $p=2$, $(\frac{2^{2a}+1}{5},3\cdot2^{2a})$ for $a\geq1$ odd. For $a=1$ the pair $(1,12)_2$ has finite monodromy. Indeed, we work with the reversed pair $(d,e)=(12,1)$ whose associated polynomial is $f(x)=x^{13}+x^{12}$ and the trace function of $\mathcal{F}_f$ is \[(k_r;s,t)\mapsto-\frac{1}{q^{r/2}}\sum_{x\in k_r}\psi_r(sx^{13}+s^{1/4}x^3+tx),\]  which is a specialization of the finite monodromy family $\FF(13,3,1;\mathbf 1)$ in \cite[Theorem 11.2.3(iii)]{KT24}.

    For $a\geq3$ consider $(x,y)=(\frac{9}{2^{2a+1}-1},\frac{4}{2^{2a+1}-1})$. Observe that $e\equiv 2^{2a}+1\pmod{2^{2a+1}-1}$, so $d+e\equiv6d\pmod{2^{2a+1}-1}$ and $9(d+e)-4\equiv2(5^2d+2d-2)=2(5\cdot 2^a+5+\sum_{k=1}^{\frac{a-1}{2}}(2^{4k-1}+2^{4k}))\equiv 1+2+2^2+2^3+\sum_{k=1}^{\frac{a-1}{2}}(2^{4k}+2^{4k+1})\pmod {2^{2a+1}-1}$, so $V_2(y-(d+e)x)=1-V_2((d+e)x-y)=1-\frac{a+3}{2a+1}=\frac{a-2}{2a+1}$. Also, $9e-4\equiv 2^{2a}+2^3+1\pmod {2^{2a+1}-1}$, so
    \[\begin{aligned}V_2(&x)+V_2(y)+V_2(y-(d+e)x)+V_2(ex-y)+V_2(-ex)=\\ &=\frac{2}{2a+1}+\frac{1}{2a+1}+\frac{a-2}{2a+1}+\frac{3}{2a+1}+\frac{2a-3}{2a+1}=\frac{3a+1}{2a+1}<\frac{3}{2}.\end{aligned}\]

\item $p=2$, $(5,\frac{2^a+1}{3})$ for odd $a\geq1$. For $a=1$ the pair has finite monodromy by \#7.
    
    For odd $a\geq3$ consider $x=y=\frac{1}{2^a-1}$. Since there are no carries while adding the binary expansions of $e$ and $4$, it follows that
    \[\begin{aligned}2V_2(&x)+V_2(x-(d+e)x)+V_2(ex-x)+V_2(-ex)\\ &= 2\frac{1}{a}+\frac{a-3}{2a}+\frac{a-1}{2a}+\frac{a-1}{2a}=\frac{3a-1}{2a}<\frac{3}{2}.\end{aligned}\]

\item $p=2$, $(\frac{2^a+1}{3},5\cdot2^a)$ for $a\geq1$ odd. For $a=1$, we have $W_2(1,10,\frac{3}{31},\frac{2}{31})=\frac{7}{5}<\frac{3}{2}$. For $a\geq3$ odd, consider $x=y=\frac{1}{2^a-1}$ and reduce to the previous item, since $5\cdot 2^a\equiv 5\pmod {2^a-1}$.

\item $p=2$, $(\frac{2^{3a}+1}{9},\frac{2^{3a}+1}{3})$ for odd $a\geq1$. For $a=1$ the pair has finite monodromy by \#7.
    
    For $a\geq 3$ consider $(x,y)=(\frac{1}{2^{3a-1}-1},\frac{5}{2^{3a-1}-1})$. Observe that $e=3d$ so $d+e-5=4d-5=4(2^2+2+1)\sum_{k=1}^{(a-1)/2}2^{3(2k-1)}-1$. Since $4(2^2+2+1)\sum_{k=1}^{(a-1)/2}2^{3(2k-1)}$ has $3(a-1)/2$ 2-adic digits equal to 1 and ends in 5 zeros, the number of 1's in the 2-adic expansion of $d+e-5$ is $(3a+5)/2$, so $V_2(y-(d+e)x)=1-V_2((d+e)x-y)=1-\frac{3a+5}{2(3a-1)}=\frac{3a-7}{2(3a-1)}$.
    
    Also, $e-5=\sum_{k=1}^{\frac{3a-1}{2}}2^{2k-1}+1-5= 2+2^2+\sum_{k=3}^{\frac{3a-1}{2}}2^{2k-1}$. Then \[\begin{aligned}V_2(&x)+V_2(y)+V_2(y-(d+e)x)+V_2(ex-y)+V_2(-ex)=\\
     &=\frac{1}{3a-1}+\frac{2}{3a-1}+\frac{3a-7}{2(3a-1)}+\frac{3a-1}{2(3a-1)}+\frac{3a-3}{2(3a-1)}=\frac{9a-5}{2(3a-1)}<\frac{3}{2}.\end{aligned}\]

     

\item $p=2$, $(\frac{2^{3a}+1}{9},2\frac{2^{3a}+1}{9})$ for $a\geq1$ odd. For $a=1$ the pair has finite monodromy by \#1.
    
    For $a\geq3$ consider $(x,y)=(\frac{3}{2^{3a+1}-1},\frac{2}{2^{3a+1}-1})$. Observe that $e=2d$ so $3(d+e)-2=2^{3a}-1$, $3e=2(1+\sum_{k=1}^{(3a-1)/2}2^{2k-1})$ and $3e-2=2\sum_{k=1}^{(3a-1)/2}2^{2k-1}$. Therefore, \[\begin{aligned}V_2(&x)+V_2(y)+V_2(y-(d+e)x)+V_2(ex-y)+V_2(-ex)=\\ &= \frac{2}{3a+1}+\frac{1}{3a+1}+\frac{1}{3a+1}+\frac{3a-1}{2(3a+1)}+\frac{3a+1}{2(3a+1)}=\frac{3a+4}{3a+1}<\frac{3}{2}\end{aligned}\] since $a\geq3$.

\item The pair $(4,3)_3$ has finite monodromy because the corresponding 
polynomial is $f(x)=x^7-x^4$ and $\mathcal{F}_f$ has trace function \[(k_r;s,t)\mapsto-\frac{1}{q^{r/2}}\sum_{x\in k_r}\psi_r(sx^7-sx^4+tx),\]  which is a specialization of the finite monodromy family $\FF(7,4,1;\mathbf 1)$ in \cite[Theorem 11.2.3(vii)]{KT24}. For the remaining pairs, we have \newline
$W_3(5,7,\frac 18,\frac 12)=W_3(10,63,\frac 18,\frac 18)=W_3(28,45,\frac 18,\frac 18)=W_3(61,12,\frac 18,\frac 18)=\frac 54$;\newline
$W_3(2,5,\frac 4{26},\frac 2{26})=W_3(4,10,\frac 2{26},\frac 2{26})=\frac 43$

\item $p=3$, $(2,3^a+1)$ for $a\geq0$. For $a=0$, we have finite monodromy by \#3.

    For $a=1$, the pair $(2,4)_3$ has finite monodromy too since the 
    polynomial associated to the reversed pair $(d,e)=(4,2)$ is $f(x)=x^6+x^5+x^4$ and the trace function of $\mathcal{F}_f$ is \[(k_r;s,t) \mapsto -\frac{1}{q^{r/2}}\sum_{x\in k_r}\psi_r(sx^5+sx^4+s^{1/3}x^2+tx),\]  which is a specialization of the finite monodromy family $\FF(5,4,2,1;\mathbf 1)$ in \cite[Theorem 11.2.3(viii)]{KT24}.

    For $a\geq 2$ we consider the reversed pair $(d,e)=(3^a+1,2)$ and $(x,y)=(\frac{\frac{3^a-1}{2}+3^{a-1}}{3^{2a}-1},\frac{3^a-1}{2(3^{2a}-1)})$. Since $\big(\frac{3^a-1}{2}+3^{a-1}\big)(d+e)-\frac{3^a-1}{2}=2\sum_{k=0}^a3^k+\sum_{k=a+1}^{2a-2}3^k+2\cdot3^{2a-1}$, $e\big(\frac{3^a-1}{2}+3^{a-1}\big)-\frac{3^a-1}{2}=\sum_{k=0}^{a-2}3^k+3^a$ and $e\big(\frac{3^a-1}{2}+3^{a-1}\big)=2\sum_{k=0}^{a-2}3^k+3^{a-1}+3^a$, we get 
    \[\begin{aligned}V_3(&x)+V_3(y)+V_3(y-(d+e)x)+V_3(ex-y)+V_3(-ex)=\\ &=\frac{a+1}{4a}+\frac{1}{4}+\frac{a-2}{4a}+\frac{1}{4}+\frac{1}{2}=\frac{6a-1}{4a}<\frac{3}{2}.\end{aligned}\] 
    
\item $p=3$, $(3^a+1,2\cdot3^a)$ for $a\geq 0$. For $a=0$ the pair has finite monodromy by \#3. For $a=1$, $W_3(4,6,\frac{11}{80},\frac{3}{8})=\frac{11}{8}<\frac 32$.

   For $a\geq 2$, multiplying by $3^a$ we get $(3^a+3^{2a},2\cdot 3^{2a})\equiv(3^a+1,2)\pmod {3^{2a}-1}$ and we reduce to the previous case by rescaling $x$ and $y$.

\item $p=3$, $(3^a+1,\frac{3^a+1}{2})$ for $a\geq 0$. For $a=0$ the pair $(2,1)_3$ has finite monodromy. Indeed, the polynomial is $f(x)=x^3-x^2$ and the trace function of $\mathcal{F}_f$ is \[(k_r;s,t)\mapsto-\frac{1}{q^{r/2}}\sum_{x\in k_r}\psi_r(-sx^2+(t+s^{1/3})x),\] which is always an algebraic integer since 2 is an FM-exponent.


    The remaining pairs have finite monodromy too since, for $a\geq 1$, the reversed pairs have associated 
    polynomial $f(x)=x^{3\frac{3^a+1}{2}}-x^\frac{3^{a+1}+1}{2}-x^{3\frac{3^{a-1}+1}{2}}+x^\frac{3^a+1}{2}$ so the trace function of $\mathcal{F}_f$ is \[(k_r;s,t) \mapsto -\frac{1}{q^{r/2}}\sum_{x\in k_r}\psi_r(-sx^\frac{3^{a+1}+1}{2}+(s+s^{1/3})x^\frac{3^a+1}{2}-s^{1/3}x^\frac{3^{a-1}+1}{2}+tx),\]  which is a specialization of the finite monodromy family $\FF(\frac{3^{a+1}+1}{2},\frac{3^a+1}{2},\frac{3^{a-1}+1}{2},1;\mathbf 1)$ in \cite[Theorem 11.2.3(i)]{KT24}.

\item $p=3$, $(1,\frac{3^a+1}{2})$ for $a\geq 0$. If $a=0$ or $1$ we have finite monodromy by \#1 and \#25. For $a\geq 2$, we show that all the reversed pairs have finite monodromy.

The corresponding polynomial is $f(x)=x^{3\frac{3^{a-1}+1}{2}}-x^\frac{3^a+1}{2}$ and the trace function of $\mathcal{F}_f$ is \[(k_r;s,t) \mapsto-\frac{1}{q^{r/2}}\sum_{x\in k_r}\psi_r(-sx^\frac{3^a+1}{2}+s^{1/3}x^\frac{3^{a-1}+1}{2}+tx),\] which is a specialization of the finite monodromy family $\FF(\frac{3^a+1}{2},\frac{3^{a-1}+1}{2},1;\mathbf 1)$ in \cite[Theorem 11.2.3(i)]{KT24}.

\item $p=3$, $(\frac{3^a+1}{2},3^a)$ for $a\geq0$. The associated polynomial is $f(x)=x^\frac{3^{a+1}+1}{2}-x^\frac{3^a+1}{2}$ and the local system $\mathcal{F}_f$ has trace function \[(k_r;s,t) \mapsto -\frac{1}{q^{r/2}}\sum_{x\in k_r}\psi_r(sx^\frac{3^{a+1}+1}{2}-sx^\frac{3^a+1}{2}+tx),\]  which is a specialization of the finite monodromy family $\FF(\frac{3^{a+1}+1}{2},\frac{3^a+1}{2},1;\mathbf 1)$ in \cite[Theorem 11.2.3(i)]{KT24}.

\item $p=3$, $(4,\frac{3^a+1}{2})$ for $a\geq 0$. For $a=0$ the pair $(4,1)_3$ has finite monodromy, the corresponding 
    polynomial being $f(x)=x^5-x^4$, so the trace function of $\mathcal{F}_f$ is \[(k_r;s,t) \mapsto -\frac{1}{\lvert k_r\rvert^\times}\sum_{x\in k_r}\psi_r(sx^5-sx^4+tx),\]  which is a specialization of the finite monodromy family $\FF(5,4,1;\mathbf 1)$ in \cite[Theorem 11.2.3(viii)]{KT24}. For $a=1$ the pair has finite monodromy by \#23.

    For $a\geq 2$ consider $x=y=\frac{1}{3^a-1}$. Since $y-(d+e)x=-\frac{2+2\cdot3+\sum_{k=2}^{a-1}3^k}{3^a-1}$, $ex-y=\frac{1}{2}$ and $ex=\frac{2+\sum_{k=1}^{a-1}3^k}{3^a-1}$ we get \[\begin{aligned}2V_3(&x)+V_3(y-(d+e)x)+V_3(ex-y)+V_3(-ex)=\\ &=\frac{1}{a}+\frac{a-2}{2a}+\frac{1}{2}+\frac{a-1}{2a}=\frac{3a-1}{2a}<\frac{3}{2}.\end{aligned}\]

\item $p=3$, $(\frac{3^a+1}{2},4\cdot3^a)$ for $a\geq0$. For $a=0$ the pair has finite monodromy by \#28. For $a=1$ we have $W_3(2,12,\frac{2}{13},\frac{2}{13})=\frac 43$. For $a\geq 2$, since $4\cdot 3^a\equiv4\pmod{3^a-1}$, we reduce to the previous item.
    
\item $p=3$, $(\frac{3^a+1}{2},\frac{3^a+1}{4})$ for odd $a\geq 1$. For $a=1$ the pair has finite monodromy by \#25.

    For $a\geq 3$ odd consider $x=\frac{1}{3^a-1},y=\frac4{3^a-1}$. Observe that $d=2e$ so $y-(d+e)x=-\frac{2+2\cdot3+3^2+2\sum_{k=2}^{\frac{a-1}{2}}3^{2k}}{3^a-1}$. Also, $ex-y=\frac{3+2\sum_{k=2}^{\frac{a-1}{2}}3^{2k-1}}{3^a-1}$. Then 
    
    \[\begin{aligned}V_3(&x)+V_3(y)+V_3(y-(d+e)x)+V_3(ex-y)+V_3(-ex)=\\ &=\frac{1}{2a}+\frac{1}{a}+\frac{a-2}{2a}+\frac{a-2}{2a}+\frac{1}{2}=\frac{3a-1}{2a}<\frac{3}{2}.\end{aligned}\]

\item $p=3$, $(\frac{3^a+1}{4},\frac{3^a+1}{4})$ for odd $a\geq1$. For $a=1$ the pair has finite monodromy by \#1.

    For odd $a\geq3$ let $x=y=\frac{2}{3^a-1}$, then $y-(d+e)x=-1$ and $ex-y=\frac{\sum_{k=1}^{a-1}3^k}{3^a-1}$, so \[\begin{aligned}2V_3(&x)+V_3(y-(d+e)x)+V_3(ex-y)+V_3(-ex)=\frac{2}{a}+0+\frac{a-1}{2a}+\frac{a-1}{2a}=\frac{a+1}{a}<\frac{3}{2}.\end{aligned}\]

\item $p=3$, $(2,\frac{3^a+1}{4})$ for $a\geq1$ odd. For $a=1$ the pair has finite monodromy by \#25.

    For odd $a\geq3$ let $x=y=\frac{1}{3^{a-1}-1}$. Since $y-(d+e)x=-\frac{2+2\sum_{k=1}^{\frac{a-1}{2}}3^{2k-1}}{3^{a-1}-1}$ and $ex-y=\frac{2\sum_{k=1}^{\frac{a-1}{2}}3^{2k-1}}{3^{a-1}-1}$ we get \[\begin{aligned}2V_3(&x)+V_3(y-(d+e)x)+V_3(ex-y)+V_3(-ex)=\\ &=\frac{1}{a-1}+\frac{a-3}{2(a-1)}+\frac{1}{2}+\frac{a-2}{2(a-1)}=\frac{3a-4}{2(a-1)}<\frac{3}{2}.\end{aligned}\] 

\item $p=3$, $(\frac{3^a+1}{4},2\cdot3^a)$ for $a\geq 1$ odd. For $a=1$ we get the pair $(1,6)_3$ with finite monodromy. Indeed, the polynomial associated to the reversed pair is $f(x)=x^7-x^6$ and the trace function of $\mathcal{F}_f$ is \[(k_r;s,t)\mapsto -\frac{1}{q^{r/2}}\sum_{x\in k_r}\psi_r(sx^7 -s^{1/3}x^2+tx),\]  which is a specialization of the finite monodromy family $\FF(7,2,1;\mathbf 1)$ in \cite[Theorem 11.2.3(vii)]{KT24}.

    For odd $a\geq3$ consider $(x,y)=(\frac{7}{3^{a+1}-1},\frac{3^a}{3^{a+1}-1})$. Since $7(d+e)-3^a=3^{a+2}+1-3^a+3(1+\sum_{k=1}^{\frac{a+1}{2}}2\cdot 3^{2k-1})\equiv 3^2-3^a+\sum_{k=1}^{\frac{a-1}{2}}2\cdot 3^{2k}\equiv 2+2\cdot3+2\sum_{k=1}^{\frac{a-1}{2}}3^{2k}+2\cdot3^a \pmod {3^{a+1}-1}$ and $7\cdot 2\cdot 3^a-3^a=13\cdot 3^a\equiv 3^a+3+1\pmod {3^{a+1}-1}$ it follows that \[\begin{aligned}V_3(&x)+V_3(y)+V_3(y-(d+e)x)+V_3(ex-y)+V_3(-ex)=\\ &=\frac{3}{2(a+1)}+\frac{1}{2(a+1)}+\frac{a-3}{2(a+1)}+\frac{3}{2(a+1)}+\frac{2a-2}{2(a+1)}=\frac{3a+2}{2(a+1)}<\frac{3}{2}.\end{aligned}\] 

\item $W_5(1,6,\frac 7{24},\frac 1{24})=W_5(2,5,\frac 7{24},\frac 1{24})=\frac{11}8<\frac 32$; \newline
$W_5(6,7,\frac 14,\frac 14)=W_5(6,15,\frac 14,\frac 14)=W_5(3,7,\frac 14,\frac 12)=\frac 54 < \frac 32$.

\item $p=5$, $(2,\frac{5^a+1}{2})$ for $a\geq0$. For $a=0$ we obtain the pair $(2,1)_5$ with finite monodromy: the associated polynomial is $f(x)=x^3-x^2$ and the trace function of $\mathcal{F}_f$ is \[(k_r;s,t)\mapsto -\frac{1}{q^{r/2}}\sum_{x\in k_r}\psi_r(sx^3-sx^2+tx),\]  which is a specialization of the finite monodromy family $\FF(3,2,1;\mathbf 1)$ in \cite[Theorem 11.2.3(ix)]{KT24}.

    For $a\geq 1$, we show that the pair $(\frac{5^a+1}{2},2)$ does not have finite monodromy. Consider $x=y=\frac{1}{5^a-1}$. Then $y-(d+e)x=-\frac{4+2\sum_{k=1}^{a-1}5^k}{5^a-1}$ and \[\begin{aligned}2V_5(&x)+V_5(y-(d+e)x)+V_5(ex-y)+V_5(-ex)=\\ &=\frac{2}{4a}+\frac{2a-2}{4a}+\frac{1}{4a}+\frac{4a-2}{2a}=\frac{6a-1}{4a}<\frac{3}{2}.\end{aligned}\]

\item $p=5$, $(\frac{5^a+1}{2},2\cdot5^a)$ for $a\geq 0$. For $a=0$ we get finite monodromy by \#35. For $a\geq 1$, since $2\cdot5^a\equiv2\pmod {5^a-1}$ we reduce to the previous case.

\item $p=7$, $(2,2).$ $W_7(2,2,\frac 13,\frac 13)=\frac 43<\frac 32$.

\end{enumerate}
\end{proof}

The following SageMath \cite{sagemath} code snippet has been used to compute the values of $W_p$:

\begin{Verbatim}[frame=single,commandchars=\\\{\}]
\textcolor{violet}{def} \textcolor{blue}{V}(p,x):
    den = x.denominator()
    num = x.numerator() % den
    \textcolor{violet}{if} p.divides(den):
        \textcolor{violet}{raise} Exception(\textcolor{red!75!black}{"denominator can't be a multiple of p"})
    r = Zmod(den)(p).multiplicative_order()
    s = \textcolor{blue!50!black}{sum}(ZZ(num*(p^r-\textcolor{green!30!black}{1})/den).digits(p))
    \textcolor{violet}{return} s/(r*(p-\textcolor{green!30!black}{1}))
\textcolor{violet}{def} \textcolor{blue}{W}(p,a,b,x,y):
    \textcolor{violet}{return} V(p,x)+V(p,y)+V(p,y-(a+b)*x)+V(p,b*x-y)+V(p,-b*x)
\end{Verbatim}


\printbibliography 

\end{document}